\documentclass[11pt]{amsart}

\usepackage[parfill]{parskip}    
\usepackage{graphicx}
\usepackage{amssymb, amsfonts,mathabx}
\usepackage{epstopdf}
\usepackage{amsmath, amsthm}
\usepackage[all,cmtip]{xy}

\usepackage{tikz}
\usetikzlibrary{matrix,arrows,patterns,calc,through,backgrounds,fadings, decorations}

\usetikzlibrary{decorations.pathreplacing}

\newtheorem{theorem}{Theorem}[section]
\newtheorem{lemma}[theorem]{Lemma}
\newtheorem*{lemma*}{Lemma}
\newtheorem{proposition}[theorem]{Proposition}
\newtheorem{corollary}[theorem]{Corollary}
\newtheorem{definition}[theorem]{Definition\rm}

\newtheorem{remark}{\it Remark\/}

\newcommand{\gitat}[1]{/\!\! /_{\hspace{-0.125em} #1}\ }
\newcommand{\biggit}{\Big/\!\!\! \!\Big/}
\newcommand{\biggitat}[1]{\Big/\!\! \!\!\Big/_{\hspace{-0.5em} #1}\ }

\newcommand{\SGS}{S_{G,{\beta} }}
\newcommand{\Aff}{\mathbb{A}}
\newcommand{\SGv}{S_G}

\newcommand{\Adj}{\text{Ad}}
\newcommand{\Tra}{\text{Tr}}
\newcommand{\End}{\text{End}}
\newcommand{\Int}{\text{int}}
\newcommand{\actson}{\reflectbox{$\righttoleftarrow$}}

\newcommand{\rea}{\Re \operatorname{e}}
\newcommand{\ima}{\Im m}

\usepackage{hyperref}

\title{On Non-Abelian Symplectic Cutting}
\author{Johan Martens}
\address[Johan Martens]{Center for Quantum Geometry of Moduli Spaces (QGM), Aarhus University, Ny Munkegade 118, bld. 1530, \r{A}rhus 8000, Denmark}
\email{jmartens@qgm.au.dk}
\author{Michael Thaddeus}
\address[Michael Thaddeus]{Department of Mathematics, Columbia University, 2990 Broadway, New York NY 10027, USA}
\email{thaddeus@math.columbia.edu}
\date{\today}  
\thanks{JM was supported by QGM (Centre for Quantum Geometry of Moduli Spaces) funded by the Danish National Research Foundation.
MT was partially supported by NSF grants DMS--0401128 and DMS--0700419.}                                         
\begin{document}

\begin{abstract}
We discuss symplectic cutting for Hamiltonian actions of non-Abelian compact groups.  By using a degeneration based on the Vinberg monoid we give, in good cases, a global quotient description of a surgery construction introduced by Woodward and Meinrenken, and show it can be interpreted in algebro-geometric terms.  A key ingredient is the `universal cut' of the cotangent bundle of the group itself, which is identified with a moduli space of framed bundles on chains of projective lines recently introduced by the authors.
\end{abstract}

\maketitle

\section{Introduction}
Since its introduction by Lerman in \cite{lerman}, 
symplectic cutting has proven to be an elementary yet remarkably useful technique in symplectic geometry with diverse applications, e.g.\ \cite{hausel,convbycuts,liruan,me}.  Symplectic cutting starts from a symplectic manifold (or orbifold) $M$ with a Hamiltonian action of a torus $T$, and a (rational) polyhedral set $P$ in $\mathfrak{t}^*$; it returns a new Hamiltonian $T$-space $M_P$ such that its image under the moment map is $\mu_T(M_{P})=\mu_T(M)\cap P$. Moreover the pre-images  $\mu_T^{-1}(\Int(P))$ in $M$ and $M_P$ are $T$-equivariantly symplectomorphic.  

The basic construction of the symplectic cut is as a global quotient  (in \cite{lerman} only actions of a single $U(1)$ were discussed; the natural generalization to cutting with arbitrary tori and polyhedral sets was given in \cite{convbycuts}).  One takes the Cartesian product $M\times \mathbb{C}^n$, where $n$ is the number of facets of $P$, and then applies a symplectic reduction by  the diagonal $U(1)^n$ action: 
\begin{equation}\label{globalquotient}M_P:= \Big( M\times \mathbb{C}^n\Big)\biggit U(1)^n.\end{equation}  
From this definition it is clear what structures $M_P$ obtains from $M$: if the reduction is generic $M_P$ will again be a symplectic orbifold with a Hamiltonian $T$-action, but it can inherit more.
 If $M$ is K\"ahler and the $T$-action extends to one of $T_{\mathbb{C}}$ then the cut space will be K\"ahler as well (though the symplectomorphism on $\mu_T^{-1}$ will not be a K\"ahler isomorphism, see also \cite{kahlercuts}); in fact if $M$ is (semi)projective the whole procedure can be understood as a geometric invariant theory quotient in algebraic geometry \cite{algcuts}.

On a topological level, one can understand this construction also more locally, 
motivating the alternative name \emph{equivariant symplectic surgery}.  One takes the pre-image of $P$ under the moment map $\mu_T$, and on the pre-images of the facets of $P$ collapses the circle subgroups of $T$ determined by the normal vectors to the facets: \begin{equation}\label{surgery}M_P\ \cong \ \mu^{-1}(P)/\sim\  =\  \bigcup_{P_I\subset P} \mu^{-1}_T(\Int(P_I))/T_I,\end{equation} where the $P_I\subset P$ are the faces of $P$, and $T_I$ is the torus perpendicular to $P_I$.   In line with this local viewpoint Lerman remarks in \cite{lerman} that symplectic cutting  can be generalized to functions that are not globally moment maps of torus actions -- this property is only needed in the pre-image of a neighborhood of the boundary of the polyhedral set.  If the torus $T$ is just $U(1)$ and one cuts with respect to two closed half-lines $P$ and $P'$ with a common boundary in $\mathfrak{u}(1)^*$, one can recover the original $M$ by applying Gompf's symplectic sum operation \cite{gompf} to the two cut spaces $M_P$ and $M_{P'}$.

It is obviously desirable to generalize this cutting construction to non-Abelian (compact) groups, and a number of approaches have appeared in the literature,  though so far with fewer applications (see however \cite{extendedfloer} for a recent use).  Notice that, given a Hamiltonian $K$-orbifold, one can always apply an Abelian cut with respect to the action of the maximal torus $T_K$, but the resulting cut space will in general only have an action by $T_K$, not by $K$.  A first construction of a non-Abelian cut  was given by Woodward in \cite{woodward}, and further detailed by Meinrenken \cite{meinrenken}.  This construction gives, for a Hamiltonian action of a compact group $K$ on a symplectic orbi\-fold $M$ and a polyhedral set $P$ in the positive Weyl chamber of $K$ that satisfies some conditions, a new space, $M_P$, whose Kirwan polytope is the intersection of $P$ and the Kirwan polytope of $M$.  

Woodward's construction is surgical in nature, in the style of (\ref{surgery}): 
compose the moment map $\mu_K$ with the quotient from $\mathfrak{k}^*$ to the positive Weyl chamber $\mathfrak{t}_+^*$, take the pre-image of $P$ under this map, and again collapse certain circle actions on the pre-images of the facets (i.e.\ locally apply an Abelian symplectic cut with respect to these circle actions).  Unlike the Abelian case, these circle actions do not extend to global actions on $M$, essentially because the map to $\mathfrak{t}^*_+$ is not smooth everywhere.  Intuitively, this explains why cutting in the non-Abelian case is a more subtle notion than in the case of torus actions.  For instance, in contrast to the Abelian case, non-Abelian cutting need not result in a K\"ahler structure on $M_P$ if $M$ was K\"ahler.  Indeed, in \cite{nonkahler} Woodward considers a co-adjoint orbit of $U(3)$ (which is of course K\"ahler), applies a non-Abelian symplectic cut with respect to the action of $U(2)\subset U(3)$, and shows, using earlier work of Tolman \cite{tolman}, that the resulting cut space does not possess any compatible K\"ahler structure.  

Besides the construction of \cite{woodward, meinrenken} two other definitions labeled symplectic cutting for non-Abelian group actions have appeared in the literature, one given by Paradan in \cite{paradan} (for general compact $K$) and one by Weitsman \cite{weitsman} (for $K=U(n)$, also discussed by Dancer and Swann in \cite{andrews}).  These authors used their constructions in the context of geometric quantization of non-compact Hamiltonian $K$-spaces with proper moment maps.  In both cases the cut  spaces were defined as symplectic reductions by $K$ of $M\times A$, where $A$ is a symplectic (in fact K\"ahler, even complex algebraic) space equipped with Hamiltonian left and right actions of $K$.  In the construction of Paradan $A$ is a projective smooth (toroidal) compactification of $K_{\mathbb{C}}$; in the construction of Weitsman 
$A=M_{n\times n}(\mathbb{C})$,  the space of $n\times n$ matrices with complex entries.  
Since symplectic reduction preserves K\"ahler structures, these symplectic cuts always result in K\"ahler spaces if $M$ is K\"ahler.  
A priori it is unclear how they are related to Woodward's construction; in fact both Paradan and Weitsman state their constructions are different.  

It is the aim of this note to show that in good cases 
a global quotient counterpart, in the style of (\ref{globalquotient}), to the construction of Woodward does exist.  As in the Abelian case this allows for the cut to be understood in K\"ahler geometry and even in algebro-geometric terms if $M$ is K\"ahler or an algebraic variety to begin with.  As a consequence it follows that the construction of Paradan is a special case of the construction of Woodward.  

In order to do this we proceed in two steps: the first involves the notion of a \emph{universal cut}, given as the symplectic cut of the group $K$ acting on its own cotangent bundle $T^*K$. 
We show that for a sufficiently general $P$ the cut space  $M_P$ can be obtained as the symplectic reduction of the Cartesian product of $M$ with this universal cut $(T^*K)_P$: \begin{equation}\label{product} M_P\cong\bigg(M\times (T^*K)_P\bigg)\biggit K.\end{equation}  This is highly reminiscent of the {symplectic implosion} construction of Guil\-le\-min, Jeffrey and Sjamaar \cite{implosion}, for which the action of the compact group on its own cotangent bundle also provided a universal implosion.  

After establishing this we can now focus our attention solely on discussing $(T^*K)_P$, which will take up the bulk of the paper.  At this point, we restrict ourselves even further to cuts where the polyhedral set is given by the intersection of a Weyl-invariant polyhedral set in $\mathfrak{t}^*$ with $\mathfrak{t}^*_+$ (with some mild extra conditions; an example is given in Figure \ref{outwardpolyt}).  Though restrictive this is still sufficient to obtain compact $M_P$ if the original moment map was proper.  In these cases we then establish a global construction for the universal cut, as a symplectic reduction or geometric invariant theory quotient of a certain affine variety.  This construction appeared recently in  other  work of the authors, \cite{us}, where $(T^*K)_P$, which in algebraic geometry is a compactification of $K_{\mathbb{C}}$, was shown to be a moduli space of $K_{\mathbb{C}}$-bundles on chains of projective lines.

\begin{figure}[h!]
\begin{center}
\begin{tikzpicture}
\begin{scope}[scale=1.5]
\fill[red!20!white] (0,0)--(2,0)--(2,.75)--(1.75,1.5)-- ++(150:.78) -- cycle;
\draw[very thick] (3,0)--(0,0)--(60:3);
\draw (2,0)--(2,.75);
\draw (2,.75)--(1.75,1.5);
\draw (1.75,1.5) -- ++(150:.75);
\draw[very thin] (2.15,0)--(2.15,.15)--(2,.15);
\draw[very thin] (1.75,1.5) ++(150:.8) ++(.243,0.028)+(150:.15) -- +(0,0) -- +(-120:.15);
\draw[blue,->] (2,.375) -- ++(.75,0);
\draw[blue,->] (1.875,1.125) -- ++(20:.75);
\draw[blue,->](1.75,1.5) ++(150:.39) -- ++(60:.75);
\end{scope}
\end{tikzpicture} 
\caption{\label{outwardpolyt} \emph{A polytope with outward normal vectors in the positive Weyl chamber, meeting any wall of the Weyl chamber perpendicularly.}}
\end{center}
\end{figure}
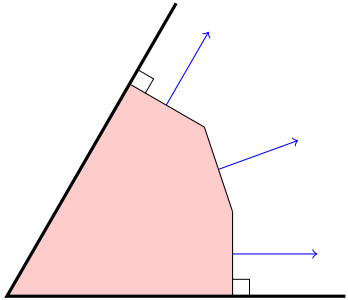

The main technical tool that allows us to do this is the remarkable \emph{Vinberg monoid}, introduced in \cite{vinberg}.  One can interpret this monoid as the total space of a particular $K_{\mathbb{C}}\times K_{\mathbb{C}}$-equivariant degeneration of $K_{\mathbb{C}}$, in such a way that the degenerate fibers now do possess the extra needed symmetry.  The simplest non-trivial example of this is $K_{\mathbb{C}}=SL(2,\mathbb{C})$.  The function to the positive Weyl chamber which one uses to apply a symplectic cut \`a la Woodward here is \begin{equation}\label{root}A\mapsto \sqrt{-\det\left(A^*A-\frac{1}{2}\Tra (A^*A)I_{2\times 2}\right)},\end{equation} the absolute value of the eigenvalue of the $SU(2)$ moment map. Since this function is not smooth everywhere on $SL(2,\mathbb{C})$,  it cannot be the Hamiltonian of a globally defined $U(1)$-action.  The Vinberg monoid in this case is simply $M_{2\times 2}(\mathbb{C})$, which we can think of as a degeneration of $SL(2,\mathbb{C})$ to the subvariety of singular matrices.  This  subvariety is singular (a cone over a quadric); however it has, unlike any of the other fibers of the determinant, an extra symmetry besides the $SU(2)\times SU(2)$-action, as it is also preserved under scalar multiplication.  On the smooth locus this global $U(1)$ action has (\ref{root}) as Hamiltonian (up to a factor $\frac{1}{2}$, see Section \ref{cox-vinberg}).

 The outline of the paper is as follows: in Section \ref{prelims} we discuss some preliminaries 
 and we describe the non-Abelian cut construction as given in \cite{woodward,meinrenken}. 
 To set the tone for the rest of the paper we also recall the Delzant construction of toric orbifolds and show that it can be interpreted as an (Abelian) symplectic cut of the cotangent bundle of the compact torus in the vein described above; this reinterpretation even extends its use. 
In Section \ref{univcut} we then restrict to cutting with respect to a universal polyhedral set and show (\ref{product}). 
In \ref{toroidal} we further restrict to the case where the (outward) normal vectors to the faces of the polyhedral set are all in the positive Weyl chamber.  We recall the \emph{Cox-Vinberg} construction introduced by the authors in \cite{us}, and show that it corresponds in symplectic geometry to the cut of $T^*K$.  This allows us to formulate in Corollary \ref{nonabglobalquotient} the generalization of (\ref{globalquotient}) to the non-Abelian case, as a torus quotient of the total space of a degeneration of $M$ based on the Vinberg monoid.
In \ref{parad} we mention how this recovers the cuts used by Paradan.  We briefly discuss the cut of Weitsman in \ref{weitsm}.  This construction, which applies to Hamiltonian $U(n)$-actions, is not a special case of Woodward's definition, but it can be described through a local surgery method which we outline.  Finally, in Appendices A and B we describe some of the symplectic geometry of complex reductive groups and reductive monoids necessary for the proof of Theorem \ref{mainprop}.
 
 \subsection*{Acknowledgements}  The authors would like to thank
Eugene Lerman, Reyer Sjamaar, Chris Woodward, Andrew Swann, Hans-Christian Herbig, Brendan McLellan, Lisa Jeffrey and, via MathOverflow, Reimundo Heluani and Peter Kronheimer for useful conversations.

\section{Preliminaries}\label{prelims}

\subsection{Notation \& basic conventions}

Let $K$ be a compact connected Lie group with Lie algebra $\mathfrak{k}$, and $G=K_{\mathbb{C}}$ its complexification, a complex reductive group. 
We fix a maximal torus $T\subset K$ with Lie algebra $\mathfrak{t}$, giving $T_{\mathbb{C}}\subset G$, and we denote the Weyl group by $W$.  We choose closed positive Weyl chambers, denoted by $\mathfrak{t}_+$ and $\mathfrak{t}_+^*$.   We use $\langle.,.\rangle$ for the pairing between $\mathfrak{t}$ and $\mathfrak{t}^*$, and we shall denote by $\tau$ the involution given by $\tau(x)=-w(x)$, where $w$ is the longest element of $W$.

Slightly adapting the terminology of Cox et al.\ \cite{coxbigbook} and
Hausel-Sturmfels \cite{haussturm}, we say a variety is {\it
  semiprojective} if it is projective over an affine variety.
We shall use symplectic reduction, 
K\"ahler quotients and geometric invariant theory (GIT) quotients, and shall denote them all by $$M\gitat{\xi}K\hspace{1cm}\text{or}\hspace{1cm}M\gitat{\xi} G,$$  where $\xi$ either indicates the central value in $\mathfrak{k}^*$ at which the symplectic reduction is taken, or the linearization used for the GIT quotient.

All the GIT quotients we encounter will have no properly semi\-stable points.  Moreover we shall always consider the orbifold (smooth Deligne-Mumford stack with trivial generic stabilizer) $\left[ M^s/G\right]$, and we shall abuse notation by still referring to this stack-theoretic quotient as the GIT quotient rather than to its coarse moduli space (which is what is normally understood as the GIT quotient).  Likewise we shall be somewhat cavalier in the symplectic category when talking about orbi\-folds and Hamiltonian group actions on them; we refer to \cite{lermanorbi} for all background.

We have actions of $K$ on itself and on $G$, which we shall denote uniformly by $\mathcal{L}_k(g)=kg$ and $\mathcal{R}_k(g)=gk^{-1}$.   We shall identify elements in $\mathfrak{k}$ with left-invariant vector fields and as such obtain identifications $$TK\cong K\times\mathfrak{k}\hspace{1cm}\text{and}\hspace{1cm}T^*K\cong K\times \mathfrak{k}^*.$$  The actions $\mathcal{L}$ and $\mathcal{R}$ of $K$ on $K$ lift to $T^*K$, and in the above trivialization they are given by $$\mathcal{L}_{\tilde{k}}(k,\lambda)=(\tilde{k}k,\lambda)\hspace{1cm}\text{and}\hspace{1cm}\mathcal{R}_{\tilde{k}}(k,\lambda)=(k\tilde{k}^{-1},\Adj^*_{\tilde{k}}\lambda) .$$
Both are Hamiltonian with respect to the canonical symplectic form,  with moment maps respectively given by $${\mu}^{\mathcal{L}}(k,\lambda)=-\Adj^*_k(\lambda)\hspace{1cm}\text{and}\hspace{1cm}{\mu}^{\mathcal{R}}(k,\lambda)=\lambda.$$

Given a matrix $A\in M_{N\times N}(\mathbb{C})$, we shall denote its Hermitian conjugate (i.e.\ conjugate transposed) by $A^*=\overline{A}^t$.  

\subsection{Labeled polytopes and stacky fans}

As mentioned before, when making a symplectic cut, Abelian or non-Abelian, we will need to specify a rational polyhedral set $P$, i.e.\ a set cut out by a finite number of half spaces determined by \begin{equation}\label{ineqs}\langle \beta_i,x\rangle\leq\xi_i,\end{equation}  where the variable $x$ ranges over $\mathfrak{t}^*$, the $\xi_i$ are real numbers, and the outward normal vectors $\beta_i$ are integral vectors in $\mathfrak{t}_{\mathbb{Z}}$.  Often one takes the $\beta_i$ to be indivisible in the integer lattice, but when working in an orbifold setting it is useful to relax this condition, and to allow the $\beta_i$ to be positive integer multiples of the minimal integral outward normal vectors to the facets of $P$.  One can indicate this by labeling the facets of $P$ with positive integers, as done in \cite{lermantolman}.  These extra data make the fan determined by $P$ into a \emph{stacky fan} as in \cite{BCS,FMN}; see Figure \ref{stacky} for an illustration.  As this creates no further complications otherwise, we shall throughout tacitly assume that such a choice of labeling or stacky fan has been made, which we shall indicate in the pictures by drawing the $\beta_i$ as normal vectors to the facets of $P$.  In the non-Abelian situation we will restrict the $P$ determined by (\ref{ineqs}) to $\mathfrak{t}^*_+$; we shall always assume that each half-space has a non-empty intersection with $\mathfrak{t}^*_+$.

\begin{figure}[h!]
\begin{center}
\begin{tikzpicture}
\begin{scope}[scale=1,xshift=-3cm]
\fill[red!20!white]
(-3,0)--(0,0)--(1.5,-1.5)--(0,-3) arc(270:180:3);
\draw[very thick](-3,0)--(0,0) --(1.5,-1.5)--(0,-3);
\draw (-1.5,0)  node[anchor=south]{$2$};
\draw (.75,-.75) node[anchor=south west]{$1$};
\draw (.75,-2.25) node[anchor=north west]{$3$};
\end{scope}

\begin{scope}[xshift=3cm,yshift=-1cm]
\fill[pattern= north west lines, pattern color = red!20!white] (-1,2.3) -- (-1,0) -- (1.3,2.3) -- cycle;
\fill[pattern= vertical lines, pattern color = red!20!white] (2.3,2.3)--(1.3,2.3) -- (-1,0) -- (2.3,-3.3)-- cycle;
\foreach \x in {-2,...,2}
\foreach \y in { -3,...,2}
{\filldraw [gray] 	(\x  , \y ) 		circle 	(2pt);}
\draw[thin](-1,0)--(-1,2.3);
\draw[thin](-1,0)--(1.3,2.3);
\draw[thin](-1,0)--(2.3,-3.3);

\draw[ultra thick, blue,->] (-1,0) -- (-1,2);
\draw[ultra thick, blue, ->] (-1,0) -- (0,1);
\draw[ultra thick, blue, ->] (-1,0) -- (2,-3);

\end{scope}
\end{tikzpicture} 

\caption{\label{stacky}\emph{The stacky fan corresponding to a labeled polyhedral set.}}
\end{center}
\end{figure}
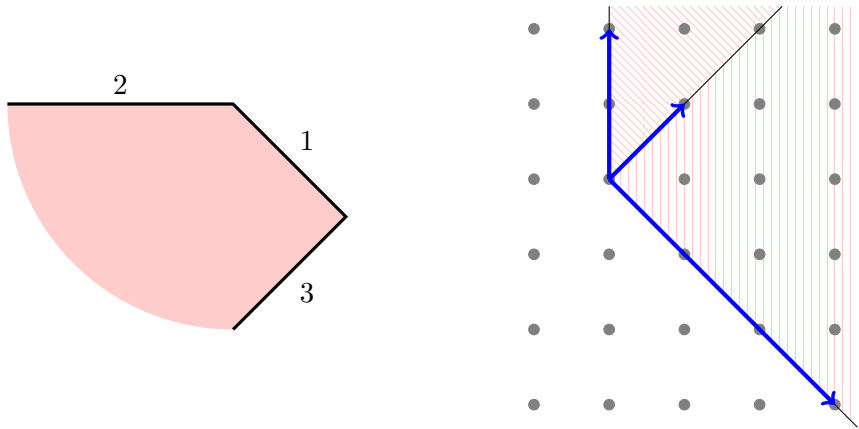

Moreover, whenever we want to interpret the cut in algebraic geometry we shall also assume that the $\xi_i$ are rational, so that we can use the $n$-tuple $\xi=(\xi_1,\ldots,\xi_n)$ to construct a fractional linearization (see e.g.\ \cite{gitflips}) with respect to which we can take a GIT quotient.  

Because we stick to an orbifold context, $P$ will always be simple (but not necessarily Delzant).  Most of what follows can however be generalized to stratified symplectic spaces, where the simple condition is no longer necessary.

\subsection{The Delzant construction as a symplectic cut of $T^*T$} \label{delzantsection}
The celebrated construction by Delzant  \cite{delzant} realizes every (compact) toric mani\-fold $M$ as a symplectic reduction of a complex vector space $\mathbb{C}^n$ (where $n$ is the number of facets of the moment polytope of $M$) by a subgroup of $U(1)^n$.  In particular, if the polytope $P$ is described by the inequalities (\ref{ineqs}), 
the $\beta_i$ determine a short exact sequence \begin{equation}\label{delzant}1\rightarrow L \rightarrow U(1)^n\rightarrow T \rightarrow 1,\end{equation} and Delzant shows that $$M\cong \mathbb{C}^n \gitat{\xi}L.$$ 
The algebro-geometric equivalent of the Delzant construction is known as the Cox construction \cite{cox}, which realizes toric varieties as categorical quotients of an open subset of $\mathbb{C}^n$ by the complexification of $U(1)^n$, and which is a GIT quotient if the toric variety is (semi)projective.

For the Delzant construction in this form to work it is crucial that the sequence (\ref{delzant}) be exact on the right (which is equivalent to saying that the $\beta_i$ generate $\mathfrak{t}$).  This is always the case for compact toric manifolds, but it often fails for toric manifolds that are non-compact but still have a proper moment map.  E.g.\ for any compact torus $T$ the cotangent bundle $T^*T$ is a toric manifold, but the (proper) moment map $$\mu_T: T^*T\cong T\times \mathfrak{t}^*\rightarrow \mathfrak{t}^*: (t, h)\mapsto h$$ is surjective, and hence there are no $
\beta_i$ at all.  One can however formulate a slight variation on the Delzant construction, which is equivalent to the Delzant construction when the sequence (\ref{delzant}) is exact on the right, but which also works for non-compact toric manifolds whose moment maps are proper onto a polyhedral set.  Indeed, we always have a $U(1)^n$-action on $\mathbb{C}^n\times T^*T$, in the usual way on the first factor and by the cotangent lift of the action $\beta:U(1)^n\actson T$ determined by the $\beta_i$ on the second.  Then one can simply use \begin{equation}\label{altdelz}
M\cong  \Big( \mathbb{C}^n\times T^*T\Big)\biggitat{\xi} U(1)^n = (T^*T)_P.\end{equation}

This variation has the additional feature that it manifestly realizes the toric manifold as an (Abelian) symplectic cut of $T^*T$.  Visually we can just interpret every factor of $U(1)^n$ cutting down the surjective image of the moment map for the action of $T$ on $T^*T$ by the corresponding half-space, finally resulting in the desired polyhedral set $P$.

\subsection{Non-Abelian symplectic cutting}\label{nonabcut}
We shall briefly review the construction given in \cite[\S8]{woodward} and \cite[\S 6]{meinrenken}.  Strictly speaking Woodward introduced the cut with respect to a single hyperplane; the natural generalization to polyhedral sets was given by Meinrenken, whose exposition we shall summarize.  

Let $M$ be a Hamiltonian $K$-orbifold, with moment map $\mu_K:M\rightarrow \mathfrak{k}^*$.  We have $\mathfrak{t}^*/W \cong \mathfrak{k}^*/K$, and we can identify $\mathfrak{t}^*/W$ with $\mathfrak{t}^*_+$, a fundamental domain for the $W$-action.  We have a canonical inclusion $\mathfrak{t}^*\hookrightarrow \mathfrak{k}^*$ as the invariant part under the coadjoint representation of $T_K$ on $\mathfrak{k}^*$, and in fact the triangle
\begin{center}
\begin{tikzpicture}
\matrix (n)[matrix of math nodes,  row sep=3em, column sep=2.5em, text height=1.5ex, text depth=0.25ex]
{\mathfrak{t}^* & & \mathfrak{k}^* \\ & \mathfrak{t}^*_+\cong \mathfrak{t}^*/W \cong \mathfrak{k}^*/K & \\ };
\path[->]
(n-1-1) edge (n-1-3) edge (n-2-2);
\path[->]
(n-1-3) edge node[auto] {$q$} (n-2-2);
\end{tikzpicture}
\end{center}
commutes.  We denote by $\Phi$ the composition \begin{equation}\label{quotient}\Phi:M\overset{\mu_K}{\longrightarrow} \mathfrak{k}^* \overset{q}{\longrightarrow} \mathfrak{t}^*_+.\end{equation}  By a theorem of Kirwan \cite{kirwanconv} $\Phi(M)$ is a polytope if $M$ is compact, often referred to as the \emph{Kirwan polytope}\footnote{By using non-Abelian cutting this implies that if $M$ is not compact but has a proper moment map then $\Phi(M)$ is a locally polyhedral set -- this was also proven (before the advent of symplectic cutting) in \cite{otherconv}.  We shall abuse terminology and still refer to $\Phi(M)$ as the Kirwan polytope.}.
Now let a polyhedral set $P$ be given in $\mathfrak{t}^*_+$, determined by inequalities (\ref{ineqs}), that satisfies a few extra properties (see Definition \ref{admissible} below). 
The basic idea is to cut with respect to the functions $f_i=\langle \beta_i,\Phi(.)\rangle$ -- i.e.\ take the pre-image of $P$ under $\Phi$ and collapse by a circle action that has $f_i$ as Hamiltonian on the $i$-th facet.  The functions $f_i$ are however not globally smooth (because $q$ is not), and therefore cannot generate global $U(1)$ actions.  The approach of \cite{woodward} and \cite{meinrenken} is therefore to just work locally and observe that this is sufficient.

In what follows, when we refer to a `face of $P$', we just mean the intersection of $P$ with a finite number of hyperplanes $\langle \beta_i,x\rangle=\xi_i$, i.e.\ the walls of the Weyl chamber itself are not taken into account.  Likewise the interior of $P$ is just given by $$\Int(P) = \mathfrak{t}^*_+ \bigcap\ \bigg\{ \langle \beta_i,x\rangle<\xi_i\bigg\};$$ in particular this can contain elements on walls of the Weyl chamber.  We shall denote faces of $P$ by $P_I$, where $I\subset\{1,\ldots,n\}$ indicates which inequalities to set to equalities.  For each such $P_I$, we shall denote by $T_I$ the group $\prod_{i\in I}U(1)_i$, which comes with a homomorphism $\phi_I$ to $T$ by the $(\beta_i)_{i\in I}$.

\begin{definition}\label{admissible} Given $M$ as above, we say $P$ is \emph{admissible} with respect to $M$ if the following three conditions are satisfied:
\begin{enumerate}
\item The affine hyperplanes $\langle \beta_i,x\rangle=\xi_i$ are all transversal (i.e.\ $P$ is simple in $\mathfrak{t}^*$).
\item For all faces $P_I$ of $P$, and all $x\in \mu_K^{-1}(P_I\cap \mathfrak{t}^*_+), \mathfrak{k}_x\cap \mathfrak{t}_I=\{ 0\}$, where $\mathfrak{k}_x$ is the Lie algebra of the stabilizer of $x$ in $K$.
\item  \label{condition} For all faces $P_I$ of $P$ meeting a face $\sigma$ of $\mathfrak{t}^*_+$ in $\Phi(M)$, the tangent space to $P_I$ contains the affine subspace perpendicular to $\sigma$.
\end{enumerate}
\end{definition}

Given a face $\sigma$ of $\mathfrak{t}_+^*$, we denote by $K_{\sigma}$ the stabilizer group for the coadjoint action, and $A_{\sigma}$ its connected center. Since $T\subset K_{\sigma}$, we always have $A_{\sigma}\subset T$. For any such $\sigma$, write 
\begin{equation}\label{YS}U_{\sigma}=\bigcup_{\sigma \subset \widetilde{\sigma}} \Int(\widetilde{\sigma})\hspace{1cm} \text{and} \hspace{1cm} Y_{\sigma}=\mu^{-1}\left(\Adj^*(K_{\sigma})(U_{\sigma})\right).\end{equation}  By the symplectic cross-section theorem (\cite[Theorem 6.1]{meinrenken},\cite[\S 26]{inphysics}), $Y_{\sigma}$ is a Hamiltonian $K_{\sigma}$-space, and the action of $A_{\sigma}$ on $Y_{\sigma}$ extends uniquely to an action on $M_{\sigma}=\Phi^{-1}\left(U_{\sigma}\right)$ that commutes with the action of $K$.  The moment map for this $A_{\sigma}$ action is given by \begin{equation}\label{characterization}\mu_{A_{\sigma}}=\pi_{\sigma}\circ \Phi,\end{equation} where $\pi_{\sigma}$ is the natural map $\mathfrak{t}^*\rightarrow\mathfrak{a}_{\sigma}^*$.  Notice that this action of $A_{\sigma}$ is in general not the induced action of a subgroup of $K$.  

Now suppose we have a $P$ admissible with respect to $M$.  
By condition (\ref{condition}) of Definition (\ref{admissible}) $\phi_I(T_I)$ is a subgroup of $A_{\sigma}$ for all $\sigma$ with $P_I\cap \sigma\neq \emptyset$.  One can choose a neighborhood $\sigma\subset V_{\sigma}\subset U_{\sigma}$ for any $\sigma$ such that \begin{equation}\label{neighchar}V_{\sigma}\cap 
P=V_{\sigma}\cap \pi_{\sigma}^{-1}\left(P\cap \mathfrak{a}_{\sigma}^*\right)\end{equation} (observe that we have canonical inclusions $\mathfrak{a}^*_{\sigma}\subset \mathfrak{t}^*$).  
By using the $A_{\sigma}$-actions we can take the (Abelian) symplectic cuts $\Phi^{-1}\left(V_{\sigma}\right)_{P_{\sigma}}$, where $P_{\sigma}=P\cap \mathfrak{a}_{\sigma}^*$.  Moreover, the $\Phi^{-1}\left(V_{\sigma}\right)$ cover $M$, and we can glue the local cuts $\Phi^{-1}\left(V_{\sigma}\right)_{P_{\sigma}}$ together to obtain a new Hamiltonian $K$-orbifold, which we refer to as the cut space $M_P$.  We have $$\Phi(M_P)=\Phi(M)\cap P,$$
and there is a decomposition into symplectic suborbifolds
\begin{equation}\label{decomp}M_P=\bigcup_{P_I\subset P} \Phi^{-1}\Big( \Int(P_I)\Big)\! \Big/T_I.\end{equation}

\begin{definition}
We say $P$ is \emph{universal} if $P$ is admissible with respect to $T^*K$, equipped with the $K$-action $\mathcal{R}$.
\end{definition}
This just means that $P$ is simple and that if a face of $P$ intersects a wall of the Weyl chamber $\mathfrak{t}^*_+$, it does so perpendicularly, as illustrated in Figure \ref{su3ex}.
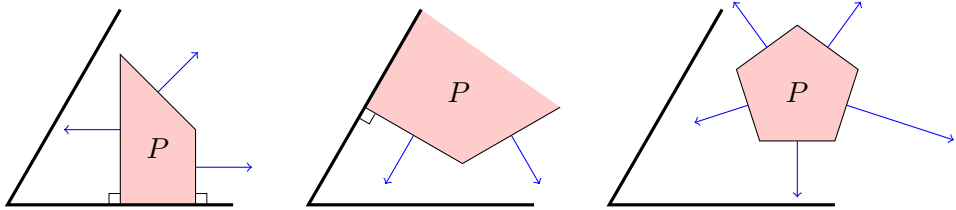
\begin{figure}[h]
\begin{center}
\begin{tikzpicture}

\begin{scope}[xshift=-4cm]
\fill[red!20!white](.5,0)+(1,0)-- +(1,2) -- +(2,1)-- +(2,0)-- cycle;

\draw[very thick] (3,0)--(0,0)--(60:3);
\draw (.5,0)+(1,0)-- +(1,2) -- +(2,1)-- +(2,0);
\draw[very thin] (2.65,0)--(2.65,.15)--(2.5,.15);
\draw[very thin] (1.35,0)--(1.35,.15)--(1.5,.15);
\draw (2,.75) node {$P$};

\draw[blue,->] (2.5,.5) -- ++(.75,0);
\draw[blue,->] (1.5,1) -- ++(-.75,0);
\draw[blue,->] (2,1.5) -- ++(45:.75);

\end{scope}

\begin{scope}
\fill[red!20!white](60:1.5)+(60:1.5)-- +(0,0)-- +(150:.25) -- ++(-30:1.5) -- +(30:1.5)--cycle;

\draw (60:1.5)-- ++(-30:1.5) -- +(30:1.5);

\draw[very thick] (3,0)--(0,0)--(60:3);
\draw[very thin](60:1.5)++(-75:.23)+(60:.15) -- +(0,0)--+(150:.15);

\draw[blue,->] (60:1.5)++(-30:.75) -- ++(240:.75);
\draw[blue,->] (60:1.5) ++(-30:1.5)++(30:.75)--+(-60:.75);
(60:1.5)++(-30:.75) -- ++(240:.75);

\draw (2,1.5) node {$P$};

\end{scope}

\begin{scope}[xshift=4cm]
\draw[very thick] (3,0)--(0,0)--(60:3);

\fill[red!20!white] (2,.85) -- ++(0:1)-- ++(72:1)--++(144:1)--++(216:1)--cycle;
\draw (2,.85) -- ++(0:1)-- ++(72:1)--++(144:1)--++(216:1)--cycle;

\draw (2.5,1.5) node {$P$};
\draw[blue,->] (2.5,.85)--+(0,-.75);
\draw[blue,->] (2,.85) ++(0:1) ++(72:.5)--+(-18:1.5);
\draw[blue,->](2,.85)  ++(0:1)++(72:1)++(144:.5)--+(54:.75);
\draw[blue,->] (2,.85) ++(0:1)++(72:1)++(144:1)++(216:.5)--+(126:.75);
\draw[blue,->] (2,.85) ++(0:1)++(72:1)++(144:1)++(216:1)++(288:.5)--+(198:.75);

\end{scope}

\end{tikzpicture}
\caption{\label{su3ex}\emph{Examples of universal polyhedral sets $P$ for $K=SU(3)$.}}
\end{center}
\end{figure}
\begin{remark}
Given $M$, not all admissible polyhedral sets need be universal.  The cut employed by Woodward in \cite{nonkahler} on a coadjoint $U(3)$-orbit (see Figure \ref{nonuniv}) cannot be made with respect to a universal polyhedral set.
\end{remark}

\begin{figure}[h]
\begin{center}
\begin{tikzpicture}

\draw[decorate, thick,decoration=brace] (2.2,3)--(2.2,0);
\draw (2.4,1.5) node[anchor=west]{$\mathfrak{t}^*_+$};
\draw (-1,.5) node[anchor=east]{$\Phi_K(M)$};
\draw[thick,->] (-1,.5) .. controls (-.75,.5) and (-.5,.5) .. (-.1,1);

\fill[red!20!white] (.8,0) +(-1.5,3)-- +(0,0) -- (2,0) -- (2,3) -- cycle;
\draw[ultra thick] (-2,0)--(2,0);

\fill[pattern=horizontal lines, pattern color = yellow!60!white](1,0)--(1.4,.4)--(-.6,2.4)--(-1,2) -- cycle;
\draw[thick] (1,0)--(1.4,.4)--(-.6,2.4)--(-1,2) -- cycle;
\draw[thick] (.8,0) +(-1.5,3)-- +(.25,-.5);
\draw[thick,blue,->] (.8,0)++(-1.35,2.7) -- +(-1,-.5);
\draw (1,2) node {$P$};

\end{tikzpicture}
\caption{\label{nonuniv}\emph{The cut employed in \cite{nonkahler}.
}}

\end{center}
\end{figure}
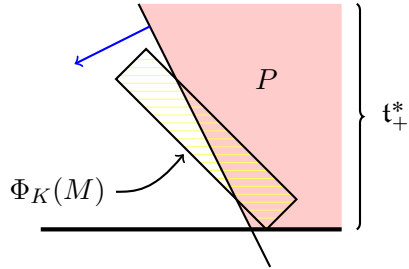

\section{A global quotient construction for cuts}

\subsection{The universal cut} \label{univcut}
We begin by introducing the notion of a universal cut, i.e.\ a cut of $T^*K$, as a tool for studying cuts for arbitrary orbifolds.  This idea is essentially applicable to any universal procedure one can apply to a Hamiltonian $K$-orbifold and was used in \cite{implosion} for symplectic implosions.
\begin{proposition}\label{universal}
Let $M$ be a Hamiltonian $K$-orbifold, and let $P$ be a universal polyhedral set in $\mathfrak{t}^*_+$, admissible with respect to $M$.  Then $$M_P\cong \Big( M\times (T^*K)_P\Big)\! \biggitat{0} K,$$  where we cut $T^*K$ with respect to the action $\mathcal{R}$, and the diagonal action of $K$ on $M\times (T^*K)_P$ uses the $\mathcal{L}$-action on the second factor.
\end{proposition}
The above choices of actions are just made for convenience.  Notice that cutting $T^*K$ using a polyhedral set $P$ and the action $\mathcal{R}$ is the same as cutting it with respect to $\tau(P)$ and the action $\mathcal{L}$.
\begin{proof}
First recall (see e.g.\ \cite[Lemma 4.8]{implosion}) that given any $M$ as above, we have that, as Hamiltonian $K$-orbifolds, $$M\cong \left( M\times T^*K\right)\gitat{0} K.$$  
To complete the proof it just suffices to observe that, given two commuting Hamiltonian actions, symplectic cutting for one (with respect to a universal and admissible polyhedral set) and symplectic reduction for the other commute.
\end{proof}
In the Abelian case $K=T$, Proposition \ref{universal}, together with (\ref{altdelz}), just states the well-known fact that the Abelian symplectic cut can be realized as the reduction by $T$ of the product of $M$ and the toric orbifold determined by $P$.

\subsection{Universal cuts and toroidal $G$-embeddings}\label{toroidal}

\subsubsection{Outward-positive polyhedral sets}

Because of Proposition \ref{universal}, we can restrict ourselves to studying $(T^*K)_P$ when cutting with respect to a universal $P$.  For the rest of the paper, we shall make a further restriction on the $P$ that we use.

\begin{definition}
A polyhedral set $P$ in $\mathfrak{t}^*_+$ determined by a finite number of inequalities of the form (\ref{ineqs}) is said to be \emph{outward-positive} 
if all the $\beta_i$ are contained in the positive Weyl chamber $\mathfrak{t}_+$.
\end{definition}
See Figure \ref{outex} for some examples of outward-positive polyhedral sets (note that none of the universal polyhedral sets of Figure \ref{su3ex} are outward-positive).

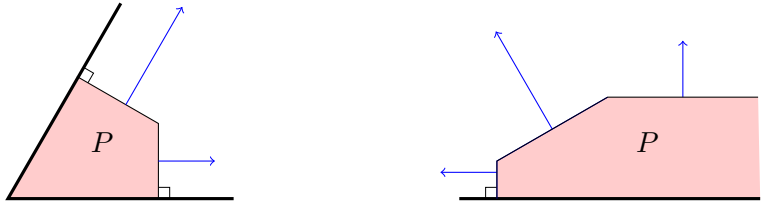
\begin{figure}[h]
\begin{center}
\begin{tikzpicture}
\begin{scope}[xshift=-4cm,scale=1]
\fill[red!20!white] (0,0)--(2,0)--(2,1)-- ++(150:1.25) -- cycle;
\draw[very thick] (3,0)--(0,0)--(60:3);
\draw (2,0)--(2,1)-- ++(150:1.25);
\draw[very thin] (2.15,0)--(2.15,.15)--(2,.15);
\draw[very thin] (2,1) ++(150:1.25) ++(.22,0.045)+(150:.15) -- +(0,0) -- +(-120:.15);
\draw[blue,->] (2,.5) -- ++(.75,0);
\draw[blue,->](2,1) ++(150:.5) -- ++(60:1.5);
\draw (1.25,.75) node {$P$};
\end{scope}

\begin{scope}[xshift=4cm]

\fill[red!20!white] (-1.5,0)-- ++(0,.5)-- ++(30:1.7)--++(0:2)--(2,0)--cycle;

\draw[very thick] (-2,0)--(2,0);
\draw[very thin] (-1.65,.15)+(.15,0)--+(0,0)--+(0,-.15);
\draw (.5,.75) node {$P$};

\draw[blue,->] (-1.5,.35) -- ++(-.75,0);
\draw[blue,->] (-1.5,0) ++(0,.5) ++(30:.85)--+(120:1.5);
\draw[blue,->]  (-1.5,0)-- ++(0,.5)-- ++(30:1.7)++(0:1)--+(0,.75);

\draw (-1.5,0) -- ++(0,.5)-- ++(30:1.7)--++(0:2);

\end{scope}

\end{tikzpicture}
\caption{\label{outex}\emph{Examples of universal outward-positive polyhedral sets $P$ for $K=SU(3)$ and $K=U(2)$.}}
\end{center}
\end{figure}

\begin{remark}  When cutting with an outward-positive polyhedral set $P$, there is little loss of generality in assuming that $P$ is also universal.  Indeed, suppose the moment map for $M$ is proper and $P$ is admissible with respect to $M$ and outward-positive, but not necessarily universal.  Then one can always impose some extra inequalities to obtain a new outward-positive polyhedral set $\widetilde{P}\subset P$ such that $\widetilde{P}$ is universal and $M_P\cong M_{\widetilde{P}}$; see Figure \ref{smaller} for an example.  This is not true if $P$ is not outward-positive.
\end{remark}

\begin{figure}[h]
\begin{center}
\begin{tikzpicture}
\begin{scope}[xshift=3cm,scale=1]

\fill[red!20!white] (2.2,0)++(110:.4)+(0,-.4)--+(0,0)--++(110:1.5)--++(150:.45)--(0,0)--cycle;

\fill[pattern=north west lines, pattern color = yellow!60!white](1.5,0)--++(60:1.5)--++(180:1.5)--cycle;
\draw  (1.5,0)--++(60:1.5)--++(180:1.5)--cycle;

\draw (2.2,0)++(110:.4)+(0,-.38)--+(0,0)--++(110:1.5)--++(150:.45);

\draw[very thick] (3,0)--(0,0)--(60:3);

\draw[very thin] (2.2,0)++(110:.4)++(.15,-.2) -- +(0,-.15)-- +(0,0) -- +(-.15,0);
\draw[very thin] (2.2,0)++(110:.4)+(0,-.4)+(0,0)++(110:1.5)++(150:.45)++(.21,.06)--+(-120:.15)--+(0,0)--+(150:.15);
\draw (-1.475,.7) node {$\Phi_K(M)$};
\draw[thick,->] (-.75,.7)..controls (-.25,.7)and (.35,.7) .. (.85,1);
\draw[thick,->] (-2.25,.7).. controls(-2.75,.7) and (-3.45,0.7).. (-3.85,1);

\draw (.5,.35) node {$\widetilde{P}$};
\draw[thick,blue,->](2.2,0)++(110:1.6)--+(20:.75);
\draw[thick,blue,->](1.3,1.93)++(-30:.17)--+(60:.75);
\draw[thick,blue,->](2.074,.25)--+(0:.75);

\end{scope}

\begin{scope}[xshift=-3cm,scale=1]

\fill[red!20!white] (2.2,0)--+(110:2.5)--(0,0)--cycle;

\fill[pattern=north west lines, pattern color = yellow!60!white](1.5,0)--++(60:1.5)--++(180:1.5)--cycle;
\draw (1.5,0)--++(60:1.5)--++(180:1.5)--cycle;

\draw (2.2,0)+(290:.2)--+(110:3);
\draw[very thick] (3,0)--(0,0)--(60:3);
\draw[thick,blue,->](2.2,0)++(110:2)--+(20:.75);

\draw (.5,.35) node {$P$};

\end{scope}

\end{tikzpicture}
\caption{\label{smaller}\emph{A symplectic cut made with a non-universal outward-positive polyhedral set $P$ and the same cut obtained with a universal outward-positive $\widetilde{P}\subset P$.}}
\end{center}
\end{figure}
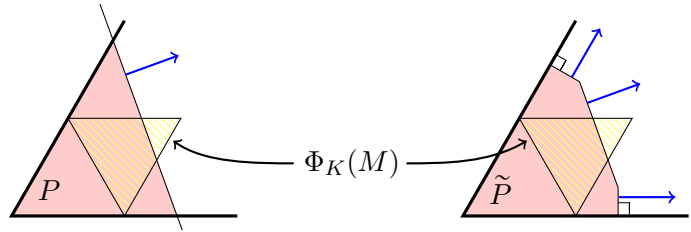

An outward-positive polyhedral set $P$ can always be written as the intersection of $\mathfrak{t}^*_+$ with a $W$-invariant polyhedral set $WP$ in $\mathfrak{t}^*$; if $P$ is moreover also universal this $WP$ will have all of its vertices in the interiors of Weyl chambers.
Given such a $P$ we shall denote its (stacky) fan of normal vectors by $\Sigma$; 
the support of $\Sigma$ is entirely contained in 
$\mathfrak{t}_+$.

 Our aim is now to show that if the polyhedral set is outward-positive 
then $(T^*K)_P$ can be understood in algebro-geometric terms as a (polarized) toroidal spherical embedding of $G$.  

These embeddings where studied in \cite{us}, where it was shown they can be interpreted as moduli spaces of framed $G$-bundles on chains of projective lines.  In \cite{us} they were denoted by $\mathcal{M}_G(\Sigma)$, but since we shall always consider them here with a choice of a (rational) polarization we shall refer to them as $\mathcal{M}_G(P)$.  These are smooth Deligne-Mumford stacks, with trivial generic stabilizer and a $G\times G$ -action.  Their coarse moduli spaces are semiprojective  toroidal spherical $G\times G$-varieties.   If $P$ is compact the $\mathcal{M}_G(P)$ are compactifications of $G$; $P$ can then also be described as the intersection of $\mathfrak{t}^*_+$ with the convex hull of the Weyl orbit of a finite number of points in the interior of $\mathfrak{t}^*_+$.

Of particular relevance here is that in \cite{us} a construction was given, dubbed the Cox-Vinberg quotient, that realizes $\mathcal{M}_G(\Sigma)$ as a torus quotient of an open subvariety of a certain affine variety.
If $\mathcal{M}_G(\Sigma)$ is semiprojective, which is always the case in our current context, the Cox-Vinberg quotient can be understood as a GIT quotient.  We shall here consider it as a symplectic reduction, in which sense it generalizes the variant on the Delzant construction outlined in Section \ref{delzantsection}.

Notice that, if $(T^*K)_P$ is a compact manifold, the Delzant conjecture, now proven in \cite{knop} and \cite{losev}, shows that $(T^*K)_P$ is determined up to equivariant symplectomorphism by its Kirwan polytope, since it is multiplicity-free.  This could be used to show the symplectomorphism we want to establish\footnote{Indeed, this strategy was used by Manolescu and Woodward for the wonderful compactification of an adjoint group in the unpublished \cite{extendedcut}.  We are grateful to the authors for sharing this manuscript with us.}.  Our strategy for the proof is entirely different however, as we aim to clarify the relationship between the construction of $(T^*K)_P$ as a symplectic cut and the Cox-Vinberg quotient.   This has as an added advantage that it also works if $(T^*K)_P$ is not compact or if it is an orbifold (or possibly even if it is singular and interpreted as a stratified symplectic space).  It seems very plausible that a generalization of the Delzant conjecture holds true  for Hamiltonian $K$-orbifolds with proper moment maps; it is well-known however that, in contrast to the Abelian case, the Delzant conjecture is false when one allows singular spaces.  

\subsubsection{The Cox-Vinberg construction}\label{cox-vinberg}

At the heart of this Cox-Vinberg quotient construction lies the Vinberg monoid $S_G$ of a (complex) reductive group $G$.  This is a reductive affine monoid, with group of units $$\widetilde{G}=(G\times Z)/ Z_G.$$  Here $Z$ is a torus with a given isomorphism to the maximal torus $T_G$, and $Z_G$ is the antidiagonal embedding of the center of $G$.  For $G$ semisimple $S_G$ was introduced by Vinberg in \cite{vinberg}, where it was called the enveloping semigroup of $G$.  For arbitrary reductive $G$ the definition was extended by Alexeev and Brion in \cite{alexeevbrion} (this generalization shares most of the properties $S_G$ has if $G$ is semisimple, with the possible exception of the universal property exhibited by Vinberg).  It can be described as follows:  by the algebraic Peter-Weyl theorem, the ring of regular functions of any complex reductive group $G$ decomposes as a $G\times G$-representation as $$k[G]=\bigoplus _{\lambda} k[G]_{\lambda},$$ where $k[G]_{\lambda}$ are the matrix coefficients of the irreducible representation with highest weight $\lambda$ as functions on $G$.  The Vinberg monoid is defined as the spectrum of a subring of $k[\widetilde{G}]$.  Let $\mathfrak{X}_G$ be the character lattice for $T_G$.  The character lattice of $T_{\widetilde{G}}$ is then given by $$\mathfrak{X}_{\widetilde{G}}=\bigg\{(x,y)\in \mathfrak{X}_G^2\Big|\ x-y = \sum n_i\alpha_i, \ n_i \in \mathbb{Z} \bigg\},$$ where the $\alpha_i$ are the positive simple roots of $G$ (or $K$).  

The Vinberg monoid is now defined (as a variety) by  $$S_G:=\text{Spec}\left( \bigoplus_{\lambda\ \in\ \mathfrak{X}_{\widetilde{G}} \cap Q_G}  k[\widetilde{G}]_{\lambda}\right),$$ where $Q_G$ is the cone $$Q_G:=\left\{(x, x+\sum_i m_i\alpha_i) \ \Big|\ x\in \mathfrak{t}^*_+, m_i\geq 0 \right\}.$$  The variety $S_G$ contains $\widetilde{G}$ as a dense open subvariety, and in fact Vinberg shows in \cite{vinberg} that $\SGv$ is a monoid, with a multiplication operation that extends the multiplication of $\widetilde{G}$.  Moreover the affine GIT quotient $$\mathbb{A}:=S_G\biggitat{0}\! (G\times G)$$ is the smooth affine toric variety for the torus $Z/Z_G$ determined by the cone spanned by the $\alpha_i$.  The fibers of $\pi_G:\SGv\rightarrow \mathbb{A}$ over the open orbit of $\mathbb{A}$ are all isomorphic to $G$ as $G\times G$-varieties; the fiber over the $Z/Z_G$-fixed point of $\mathbb{A}$ is referred to as the \emph{asymptotic semigroup} of $G$ by Vinberg \cite{vinberg2}.

Suppose now that a polyhedral set as above is given, with (outward) normal vectors $\beta_i$.  These $\beta_i$ determine homomorphisms $\phi_{\beta_i}$ from $\mathbb{G}_m=\mathbb{C}^*$ into $Z$ and hence also into $Z/Z_G$.  Since moreover all the $\beta_i$ are contained in the positive Weyl chamber $\mathfrak{t}_+$, the collective homomorphism $\phi_{\beta}$ from $\mathbb{G}_{\beta}:=\mathbb{G}_m^n$ to $Z/Z_G$ extends to a homomorphism of monoids $\overline{\phi}_{\beta}$ from $\mathbb{A}_{\beta}:=\mathbb{C}^n$ to $\mathbb{A}$.  We can now take the fibered product 
\begin{center}
\begin{tikzpicture}
\matrix (n)[matrix of math nodes,  row sep=3em, column sep=2.5em, text height=1.5ex, text depth=0.25ex]
{\SGS & \SGv \\ \Aff_{\beta} &  \Aff.\\};
\path[->] 
(n-1-1) edge (n-1-2) edge  
(n-2-1);
\path[->] 
(n-2-1) edge  node[auto] {$\overline{\phi}_{\beta}$} (n-2-2);
\path[->] 
(n-1-2) edge  node[auto] {$ \pi_G $} (n-2-2);

\end{tikzpicture}
\end{center}

This is a monoid with group of units $G\times \mathbb{G}_{\beta}$, flat over $\mathbb{A}_{\beta}$, with generic fibers isomorphic to $G$.  Note that $\overline{\phi}_{\beta}$ can only be the identity if $K$ is adjoint and the $\beta_i$ are the fundamental coweights $\varpi_i^{\vee}$.  This is also the reason behind the extra factor occurring in the $SL(2,\mathbb{C})$ example given in the introduction: the $S_{SL(2,\mathbb{C}),\beta}$ one wants to consider to cut $SL(2,\mathbb{C})$ is not just $M_{2\times 2}(\mathbb{C})$ (which is $S_{SL(2,\mathbb{C})}$), but rather matrices in $M_{2\times 2}(\mathbb{C})$ together with a choice of a square root of their determinant.

It is straightforward to check that we can describe $\SGS$ directly as follows.

\begin{lemma}
If we define the cone \begin{equation}\label{cone}Q_{G,\beta}:=\left\{\left(x, (\langle \beta_i,x\rangle+ m_i)_i\right) \ \Big|\ x\in \mathfrak{t}^*_+, m_i\geq 0 \right\}\subset \mathfrak{t}^*\oplus (\mathfrak{u}(1)^n)^*,\end{equation}
we have $$\SGS= \text{Spec}\left(\bigoplus_{\lambda \in \mathfrak{X}_{G\times \mathbb{G}_{\beta}}  \cap Q_{G,\beta}} k[G\times \mathbb{G}_{\beta}]_{\lambda}\right).$$
\end{lemma}

Consider now the action of $\mathbb{G}_{\beta}$ on $\SGS$, given by the obvious action on $\mathbb{A}_{\beta}$ and the action on $S_G$ induced by the homomorphism $\mathbb{G}_{\beta}\rightarrow Z$. It is proven in \cite{us} that  we have 
\begin{equation}\label{isommodul}\SGS \gitat{\xi} \mathbb{G}_{\beta}\cong\mathcal{M}_G(P);\end{equation}  for our purposes we can take this as the definition of $\mathcal{M}_G(P)$.  

\subsubsection{Hamiltonian geometry of $\SGS$}
We want to show that, if we interpret this quotient in symplectic geometry, the result is symplectomorphic to the corresponding cut of $T^*K$.  In order to do so we need to choose a symplectic structure on $\SGS$;  
we shall choose one coming from an affine embedding by restricting the Euclidean metric on the ambient space. 

In particular, we shall choose an embedding $\iota$ that realizes $\SGS$ as a closed submonoid of some $M_{N\times N}(\mathbb{C})$. This of course determines a representation of its group of units, $G\times \mathbb{G}_{\beta}$, and $\iota$ is equivariant with respect to the left and right action of $G\times \mathbb{G}_{\beta}$ on $S_{G,\beta}$ by multiplication.  Such a choice of $\iota$ is always possible: see \cite[Remark on page 169]{vinberg}.
In this way
$\SGS$ becomes K\"ahler\footnote{Strictly speaking, since $\SGS$ a priori might have singularities, it would have to be interpreted as a stratified symplectic space.  We can ignore these issues however, since the GIT-stable subvariety of $\SGS$ is smooth, by \cite[Theorem 7.1]{us} and \cite[Theorem 8]{vinberg}.}, and the GIT construction will lead to a variety that is semiprojective 
and hence again K\"ahler by combining the Fubini-Study metric with the Euclidean one.  We shall use a maximal compact subgroup $K\times U(1)^n$ of $G\times \mathbb{G}_{\beta}$ compatible with this K\"ahler structure. 

Notice that our only ambition here is to relate a purely algebraic construction to a purely symplectic one, and we can therefore just choose a compatible K\"ahler structure to serve our purposes. The full problem of describing all of the admissible K\"ahler structures that can occur on this symplectic cut is more subtle and is not addressed here.

It follows from \cite[Theorem 4.9]{sjamaar} 
that the cone $Q_{G,\beta}$ given in (\ref{cone}) 
is the image of $\Phi_{K\times U(1)^n}$, i.e.\ the moment map $\mu_{K\times U(1)^n}$ for the $\mathcal{R}$-action of $K\times U(1)^n$ on $\SGS$ composed with the projection on the positive Weyl chamber $\mathfrak{t}^*_+\oplus (\mathfrak{u}(1)^n)^*$ (the map $\Phi_{K\times U(1)^n}$ for the $\mathcal{L}$-action is given by composing the one for the $\mathcal{R}$-action with $\tau$).

Lemma \ref{sectionconstr} gives us moreover a section of the moment map $\mu_{K\times U(1)^n}$, which we shall simply denote by $s$.  In what follows we shall furthermore denote the moment map for the  $U(1)^r$-action on $\SGS$ by $\mu$.

\subsubsection{Correspondence results}
With all of this set up, we are ready for our main result. 

\begin{theorem}\label{mainprop}
There is a $K\times K$-equivariant symplectomorphism of (real) orbifolds $$ (T^*K)_P\cong \mathcal{M}_G(P).$$
\end{theorem}

\begin{proof}

We begin by establishing local (orbifold) diffeomorphisms.  We shall follow here the notation used in Section \ref{nonabcut}.  Recall that there local descriptions of the non-abelian cut were given as abelian symplectic cuts $\Phi^{-1}(V_{\sigma})_{P_{\sigma}}$ using the local action of $A_{\sigma}$ that commuted with the action of $K$.  In turn these abelian symplectic cuts are defined as $$\Phi^{-1}(V_{\sigma})_{P_{\sigma}} =  \Big(\Phi^{-1}(V_{\sigma})  \times \mathbb{C}^r\Big) \biggitat{(\xi_j)_{j\in J}} U(1)^r,$$ where $J$ denotes the equations among all inequalities (\ref{ineqs}) needed to describe $P_{\sigma}$ -- for convenience we shall assume that $J=\{1,\dots,r\}$.  In what follows we shall need to make one small modification to this description: rather than using $V_{\sigma}$ that are neighborhoods of all of $\sigma$ in $U_{\sigma}$ satisfying (\ref{neighchar}), we shall just use neighborhoods $\widetilde{V}_{\sigma}$ of $\sigma \cap P$ in $U_{\sigma}$ that satisfy (\ref{neighchar}) and moreover have the property that the inequalities (\ref{ineqs}) are strict for all $j\notin J$, as illustrated in Figure \ref{vsigma}.  One easily sees that $\Phi^{-1}(V_{\sigma})_{P_{\sigma}}=\Phi^{-1}(\widetilde{V}_{\sigma})_{P_{\sigma}}$ if $V_{\sigma}\cap P=\widetilde{V}_{\sigma}\cap P$.

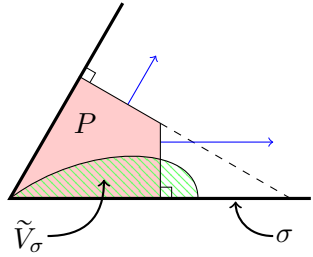
\begin{figure}[h]
\begin{center}
\begin{tikzpicture}
\fill[red!20!white] (0,0)--(2,0)--(2,1)-- ++(150:1.25) -- cycle;
\draw[dashed] (2,1) -- ++(-30:2);
\draw[very thin] (2,1) ++(150:1.25) ++(.22,0.045)+(150:.15) -- +(0,0) -- +(-120:.15);
\draw[blue,->] (2,.75) -- ++(1.5,0);
\draw[blue,->](2,1) ++(150:.5) -- ++(60:.75);
\draw (1,1) node {$P$};
\draw (3.4,-.5) node[anchor=west] {$\sigma$};
\draw[thick,->] (3.5,-.5) .. controls (3.15,-.5) and (3,-.25) .. (3,-.04);

\filldraw[pattern=north west lines, pattern color = green!60!white]
(0,0)
.. controls (1,.75) and (2.5,.75) .. (2.5,0) -- cycle;
\draw (2,0)--(2,1)-- ++(150:1.25);
\draw[very thin] (2.15,0)--(2.15,.15)--(2,.15);

\draw[very thick] (4,0)--(0,0)--(60:3);

\draw (.6,-.5) node[anchor=east] {$\widetilde{V}_{\sigma}$};

\draw[thick,->] (.5,-.5) .. controls (1,-.5) and (1.25,-.25) .. (1.25,.25);

\end{tikzpicture}
\caption{\label{vsigma}\emph{Example of $\widetilde{V}_{\sigma}$ for $K=SU(3)$.}}
\end{center}
\end{figure}

We shall now write down a map $T_{\sigma}$ from $\Phi^{-1}(\widetilde{V}_{\sigma})  \times \mathbb{C}^r$ to $\SGS$, for which we describe elements in $\mathbb{C}^r$ in polar coordinates as $(e^{i\theta_j}r_j)_{j\in J}$ -- we normalize the moment map of $U(1)^r$ on $\mathbb{C}^r$ to be $(r_j^2)_j$.  We define $$T_{\sigma}:((k,\gamma),(e^{i\theta_j}r_j)_j)\mapsto (k,\underbrace{(e^{i\theta_j})}_{j\in J},\underbrace{(1,\dots,1)}_{j\notin J}).s\big(\gamma, \underbrace{(r_j^2+\langle \beta_j, q(\gamma)\rangle)}_{j\in J},\underbrace{(\xi_j)}_{j\notin J}\big),$$ where `$.$' denotes the $\mathcal{L}$-action; see Figure \ref{fetish} for an illustration.  This is well-defined: because of the restrictions to $\widetilde{V}_{\sigma}$ the argument of $s$ indeed lies in its domain, i.e.\ $\Adj^*(K\times U(1)^n)(Q_{G,\beta})$, and secondly, there is an ambiguity in the $e^{i\theta_j}$ whenever $r_j=0$, but one checks  that the action of $e^{i\theta_j}$ on both $\mathbb{A}_{\beta}$ and (using \cite[Theorem 7]{vinberg}) on $\SGv$, and hence on the image of $T_{\sigma}$, is trivial whenever $r_j=0$.

\begin{figure}[h]
\begin{center}
\begin{tikzpicture}

\begin{scope}[scale=1.2]

\fill[pattern=north west lines, pattern color = blue!20!white](2,2)--(0,0)--(2.828,0) arc(0:45:2.828);

\draw[very thin,->] (0,0)--(3.5,0);
\draw(0,1.75) node[anchor=east]{$\Phi_{K}$};
\draw(3.25,0) node[anchor=north]{$\mu$};

\draw[very thick] (2,2) -- (0,0) -- (2.828,0);

\draw[very thin,->] (0,0)--(0,2);

\draw[very thick, red!40!white] (1,0)--(1,1);
\filldraw[red!40!white] (1,0) circle (2pt);
\filldraw[red!40!white] (1,1) circle (2pt);
\filldraw[white] (1,1) circle (1.5pt);

\draw[very thin] (-4,0)--(-4,2);
\draw[very thick, red!40!white] (-4,0)--(-4,1);
\filldraw[red!40!white] (-4,0) circle (2pt);
\filldraw[red!40!white] (-4,1) circle (2pt);
\filldraw[white] (-4,1) circle (1.5pt);

\draw[decorate, thick,decoration=brace] (-4.1,0)--(-4.1,1);

\draw(-4.3,.5) node[anchor=east]{$\widetilde{V}_{\sigma^{min}}$};
\draw(-4,1.75) node[anchor=west]{$\mathfrak{t}^*_+$};

\draw [thick,->](-3.9,0.5) .. controls (-1.9,1) and (-1.1,1) ..
node[above] {$T_{\sigma^{min}}$}
 (.9,.5);
 \draw(2.2,.7) node{$Q_{G,\beta}$};
 \draw(1,0) node[anchor=north]{$\xi$};
\end{scope}

\begin{scope}[scale=1.2,yshift=-3.5cm]
\fill[pattern=horizontal lines, pattern color = red!40!white](2,2)--(0,0)--(2.828,0) arc(0:45:2.828);

\draw[dashed] (1,0)--(1,1);
\draw[very thin,->] (0,0)--(3.5,0);
\draw(0,1.75) node[anchor=south west]{$\Phi_{K}$};
\draw(3.25,0) node[anchor=north]{$\mu$};

\draw[very thick] (2,2) -- (0,0) -- (2.828,0);

\draw[very thin,->] (0,0)--(0,2.2);

\draw[very thick, red!40!white] (0,0)--(2,2);
\filldraw[red!40!white] (0,0) circle (2pt);
\filldraw[white] (0,0) circle (1.5pt);

\draw[very thin] (-6,0)--(-6,2);

\fill[pattern=horizontal lines, pattern color = red!40!white](-6,0.05)--(-3,0.05)--(-3,2) -- (-6,2) -- cycle;

\draw(-4.1,.75) node{$\widetilde{V}_{\sigma^{max}}\times [0,\infty)$};
\draw(-6,1.75) node[anchor=east]{$\mathfrak{t}^*_+$};

\draw [thick,->](-3.9,1.2) .. controls (-1.9,1.7) and (-.3,1.7) ..
node[above] {$T_{\sigma^{max}}$}
 (1.7,1.2);
 \draw(2.2,.7) node{$Q_{G,\beta}$};
 \draw(1,0) node[anchor=north]{$\xi$};
\draw[dashed] (-6,1)--(-5.05,0.05);
\draw[very thick, red!40!white] (-6,0)--(-6,2);
\filldraw[red!40!white] (-6,0) circle (2pt);
\filldraw[white] (-6,0) circle (1.5pt);

\end{scope}

\end{tikzpicture}
\caption{\label{fetish}\emph{The maps $T_{\sigma}$ on the level of Kirwan polytopes for  $K=SU(2)$, $\mathfrak{t}^*_+\cong [0,\infty)\subset \mathbb{R}$, $P=[0,\xi ]$, $\sigma^{min}=\{ 0\}$,\ $\widetilde{V}_{\sigma^{min}}=[0,\xi )$, $\sigma^{max}=(0,\infty)$,  $\widetilde{V}_{\sigma^{max}}=(0,\infty)$.  The dashed lines in the lower picture indicate level sets of the moment maps for the $U(1)^r$ and $U(1)^n$-actions respectively.}}
\end{center}
\end{figure}
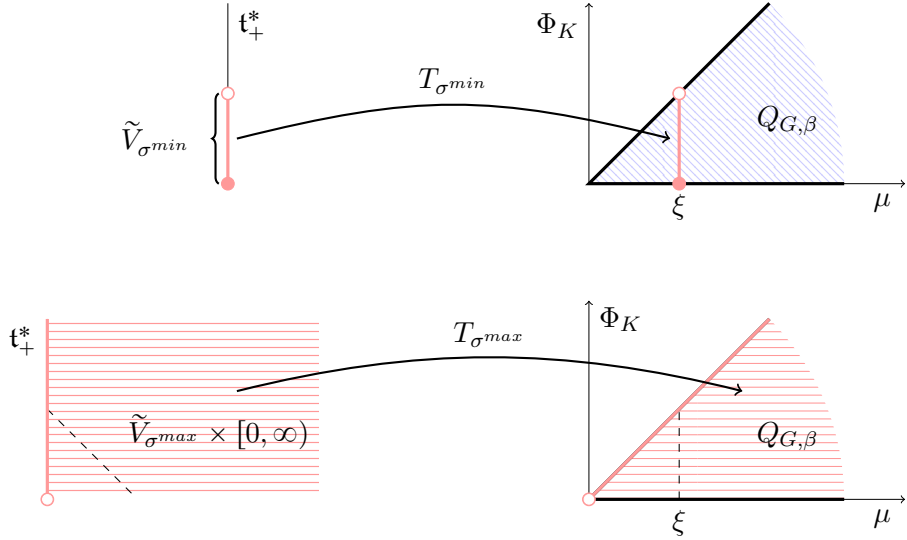

The map $T_{\sigma}$ is equivariant for the action of $K\times K$ and $U(1)^r$, where $K\times K$ acts in the obvious way on $\Phi^{-1}(\widetilde{V}_{\sigma})$ and $U(1)^r$ acts diagonally on the product, on  $\Phi^{-1}(\widetilde{V}_{\sigma})$ through $A_{\sigma}$ and on $\mathbb{C}^r$ in the obvious way. Using (\ref{characterization}) one sees that the moment map for this $U(1)^r$-action on $\Phi^{-1}(\widetilde{V}_{\sigma})  \times \mathbb{C}^r$ is given by $$((k,\gamma),(e^{i\theta_j}r_j)_j)\mapsto (r_j^2 + \langle \beta_j,q(\gamma)\rangle)_j;$$ from this and the formula for $T_{\sigma}$ it follows that
the $(\xi_1,\dots,\xi_r)$-level set of this moment map  gets sent by $T_{\sigma}$ to  $\mu^{-1}(\xi)\subset \SGS$.  
Using the equivariant normal forms (see \cite{marle,guilsternconf}) for said actions one sees that $T_{\sigma}$ is an embedding as well.

One can now think of the symplectic reduction $\SGS\gitat{\xi} U(1)^n$ as happening in two stages: first take the reduction by the first $r$ factors of $U(1)^n$, and then by the remaining $n-r$ ones.  If the image of $T_{\sigma}$ intersects an orbit for the last $n-r$ $U(1)$-factors it does so transversely and in a single point (the latter follows by using Lemma \ref{crucial} and \cite[Theorem 7]{vinberg}); as a result $T_{\sigma}\big(\Phi^{-1}(\widetilde{V}_{\sigma})  \times \mathbb{C}^r\big)\gitat{(\xi_j)_{j\in J}} U(1)^r$ provides an orbifold slice for the level set of the moment map for the last $n-r$ $U(1)$-factors acting on $\SGS\gitat{(\xi_j)_{j\in J}} U(1)^r$.  Therefore the induced map $$\widetilde{T}_{\sigma}:\Phi^{-1}(\widetilde{V}_{\sigma})_{P_{\sigma}}\rightarrow\widetilde{T}_{\sigma}\Big(\Phi^{-1}(\widetilde{V}_{\sigma})_{P_{\sigma}}\Big)\subset \SGS  \gitat{\xi} U(1)^n$$ is an orbifold diffeomorphism onto an open suborbifold of $\mathcal{M}_G(P)$.   

The various $\widetilde{T}_{\sigma}$ glue together to $\widetilde{T}$, which is still an embedding.  Indeed, if we think of $(T^*K)_P$ on a topological level as starting from $\Phi^{-1}(P)$ and collapsing the boundary following (\ref{decomp}), then as a map between topological spaces $\tilde{T}$ is induced by  the map \begin{equation*}T^{top}:\Phi^{-1}(P)\rightarrow \mu^{-1}(\xi)\subset \SGS: (k,\gamma)\mapsto k.s(\gamma, \xi),\end{equation*} from which one sees that $\widetilde{T}$ is injective since $T^{top}$ is injective and any $T_I$-orbit in its source that has to be collapsed gets sent into a $U(1)^n$-orbit.    

To see that $\widetilde{T}$ is also surjective, observe that 
every $U(1)^n$ orbit in $\mu^{-1}(\xi)\subset \SGS$ meets some $T_{\sigma}(\Phi^{-1}(\widetilde{V}_{\sigma}))$, essentially since it follows from (\ref{cone}) that $$\Phi_K\left(\mu^{-1}(\xi)\right)=P,$$ where $\Phi_K$ is $q\circ \mu_K:\SGS\rightarrow \mathfrak{t}^*_+$.
Hence $\widetilde{T}$ is a global diffeomorphism between the orbifolds. 

To show that this is a symplectomorphism, it suffices to show it for a dense open subset.  Let $\sigma^{min}$ be the minimal face of $\mathfrak{t}^*_+$ (containing the origin).  We can take as $\widetilde{V}_{\sigma^{min}}$ the whole of $\Int(P)$.
We have $$\Big(T_{\sigma^{min}}(\Phi^{-1}(\Int(P)))\times U(1)^n\Big) = \Big(\mu^{-1}(\xi)\cap (G\times \mathbb{G}_{\beta}\subset \SGS)\Big).$$  Since the symplectic form on $\SGS$ is obtained from $\iota$ and we are now working in the reductive group $G\times \mathbb{G}_{\beta}$ inside $\SGS$, Lemma \ref{sympl2} ensures that we can use Lemma \ref{sympl}, from which it follows that in the diagram
\begin{center}
\begin{tikzpicture}
\matrix (n)[matrix of math nodes,  row sep=3em, column sep=2.5em, text height=1.5ex, text depth=0.25ex]
{T_{\sigma^{min}}(\Phi^{-1}(\Int(P)))\times U(1)^n &\mu^{-1}(\xi)\subset \SGS \\ \Phi^{-1}(\Int(P)) &\\ };
\path[right hook->]
(n-1-1) edge (n-1-2);
\path[->]
(n-1-1) edge (n-2-1);
\end{tikzpicture}
\end{center}
the pull-backs of the symplectic forms on $\Phi^{-1}(\Int(P))$ and $\SGS$ coincide, which establishes that the global diffeomorphism $\widetilde{T}$ is indeed a symplecto\-morphism on $\Phi^{-1}(\Int(P))$, and hence everywhere.   This concludes the proof.
\end{proof}

Finally from this we can conclude a non-Abelian version of (\ref{globalquotient}).
\begin{corollary}\label{nonabglobalquotient}
There is a natural symplectomorphism \begin{equation}\label{vinbglobquot}M_P\cong \bigg(\Big(\left( M \times \SGv\right)\biggitat{0} K\Big) \times_{\Aff} \mathbb{A}_{\beta}\bigg)\biggitat{\xi} U(1)^n.\end{equation}
\end{corollary}
\begin{proof}
It suffices to combine Theorem \ref{mainprop} with Proposition \ref{universal} and (\ref{isommodul}), and notice that since all group actions involved commute we can switch the order of the quotients.
\end{proof}

\begin{remark}
If $K=T$ is a torus, $\SGv$ is just $G=T_{\mathbb{C}}$ and $\mathbb{A}$ is a point;  (\ref{vinbglobquot}) then just reduces to (\ref{globalquotient}).  If $M$ is a semiprojective algebraic variety, we can think of $\left(M\times \SGv\right) \gitat{0} G$ as the total space of a flat degeneration of $M$ over $\mathbb{A}$; this family also appeared in \cite[\S7]{alexeevbrion}.
\end{remark}

\section{Comparison with constructions of Paradan and Weitsman}
\subsection{Paradan}\label{parad}
As stated in the introduction, Paradan \cite{paradan} defines non-Abelian symplectic cutting of $M$ as the reduction $$M_{\text{cut}}=(M\times A)\gitat{0} K,$$ where $A$ is a smooth projective variety which $G\times G$-equivariantly compactifies $G$.  Paradan constructs this by taking a finite collection of irreducible representations $V_{\lambda_i}$ of $G$, where all the highest weights $\lambda_i$ are \emph{regular} dominant weights, i.e.\ contained in the interior of $\mathfrak{t}^*_+$.  He shows that if the convex hull of the $W$-orbits of the $\lambda_i$ is a Delzant polytope then one has an embedding $$G \hookrightarrow \mathbb{P}\left( \bigoplus_i \End(V_{\lambda_i})\right),$$ and $A$, the closure of $G$ in this projective space, is a non-singular variety.  

It follows from the theory of spherical embeddings (see e.g.\ \cite[Example 29.7]{timashev} or \cite{pezzini}) that all such $A$ are toroidal spherical $G\times G$-varieties.  If $K$ is adjoint or $Sp(n)$ (or a product of these), it even suffices to take a single regular $\lambda_i$.  In this case $A$ will have a unique closed $G\times G$ orbit, i.e.\ it will be a \emph{wonderful} compactification of $G$ since it is also smooth.

Any toroidal embedding of $G$, in particular the non-singular ones used by Paradan, can be obtained from the Cox-Vinberg construction given in Section \ref{cox-vinberg}. It suffices to use as $P$ the intersection of the convex hull of the $W$-orbits of the $\lambda_i$ with $\mathfrak{t}^*_+$ -- it is easy to see this is a universal polytope.  As a result the cut construction of Paradan can be seen as a special case of the cut of Woodward.

\subsection{Weitsman}\label{weitsm}
In \cite{weitsman} Weitsman defines a symplectic cut for Hamiltonian $U(n)$-manifolds by
$$M_{\epsilon}=\Big(M\times M_{n\times n}(\mathbb{C})\Big)\biggitat{\epsilon} U(n),$$ where $\epsilon$ is a central value in $\mathfrak{u}(n)^*$.  A surgery-type description for this construction was also given in \cite[\S 3]{andrews}; we rephrase it here in the language we have used for the Woodward construction.  We shall use as $T$ the set of diagonal matrices in $U(n)$, as before we identify $\mathfrak{u}(n)$ with $\mathfrak{u}(n)^*$ by means of an invariant metric, and we use as positive Weyl chamber $$\mathfrak{t}^*_+\ =\ \left\{  i\left(\begin{array}{cccc} \lambda_1 & & & \\ & \lambda_2 & &\\ & & \ddots & \\  & & & \lambda_n \end{array}\right) , \hspace{.5cm} \text{with} \hspace{.5cm} \lambda_1\geq\dots \geq \lambda_n\right\}.$$  Define the subsets of $\mathfrak{t}^*_+$ $$P_{\epsilon, k}:= \big\{ \lambda_1\geq \dots \geq \lambda_{n-k}>\lambda_{n-k+1}=\dots=\lambda_n=\epsilon\big\}.$$  Weitsman calls $M$ \emph{cuttable at} $\epsilon$ if $\epsilon$ is a regular value of the moment map on $M\times M_{n\times n}(\mathbb{C})$, one then has (cf.\ \cite[Remark 2.9]{weitsman}) $$M_{\epsilon}=\bigcup_{k\in\{0,\dots, n\}} \Phi^{-1}(P_{\epsilon,k})\Big/\sim_k,$$  where $\sim_k$ is determined by dividing out by the $\mathcal{R}$-stabilizer of $\sqrt{\frac{1}{i}\mu_K(.)}\in M_{n\times n}(\mathbb{C})$ -- all of these are isomorphic to $U(k)$.  A major difference with the Woodward construction is that the groups one quotients out by on the boundary of the polytope $\overline{P_{\epsilon,0}}$ are non-Abelian.  In particular, $M\gitat{\epsilon}U(n)$ itself is a subspace of $M_{\epsilon}$.  It is also harder to see if $M$ is cuttable, compared to the conditions of Definition (\ref{admissible}) for the Woodward construction.

The main aim of the cut in \cite{weitsman} is to produce compact spaces out of non-compact $U(n)$-spaces with proper moment maps.  For this purpose a single cut will in general not suffice.  Rather Weitsman takes a cut, reverses the symplectic structure, takes another cut, and reverses the symplectic structure again, to obtain $\overline{\left( \overline{M_{\epsilon}} \right)_{\delta}}$, which will always be compact.

\appendix

\section{Symplectic structures on complex reductive groups}\label{sympred}

As we want to compare a symplectic construction involving $T^*K$ with an algebraic construction involving $G$, we need a way to relate the two.   This is provided by the following:

\begin{lemma}\label{sympl}
Let $G$ be a connected complex reductive group,  the complexification of a compact Lie group $K$.  Assume we have a symplectic structure $\omega_G$ on $G$ such that the two actions $\mathcal{L}$ and $\mathcal{R}$ of $K$ on $G$ are Hamiltonian with moment maps $\widetilde{\mu}^{\mathcal{L}}$ and $\widetilde{\mu}^{\mathcal{R}}$.  Assume further that we have a projection $\Pi:G\rightarrow K$ that is equivariant for $\mathcal{L}$ and $\mathcal{R}$, and whose fibers are Lagrangian for $\omega_G$.  Then the morphism $$\Psi:G\rightarrow K\times\widetilde{\mu}^{\mathcal{R}}(K)\subset T^*K :g\mapsto\big(\Pi(g),\widetilde{\mu}^{\mathcal{R}}(g)\big)$$ is a $K\times K$-equivariant symplectomorphism onto an open submanifold of $T^*K$. 
\end{lemma}  In particular, if $\widetilde{\mu}^{\mathcal{R}}$ is surjective, we have a symplectomorphism with all of $T^*K$.  As $G/K$ is contractible, this image always contracts onto $K$.  Notice that the use of $\widetilde{\mu}^{\mathcal{R}}$ in the definition of $\Psi$ is a consequence of the choice we made to identify $T^*K$ with $K\times \mathfrak{k}^*$ by means of left-invariant vector fields.
\begin{proof}
Observe first that
since the two actions commute, $\widetilde{\mu}^{\mathcal{L}}$ is invariant for the $\mathcal{R}$-action and vice versa.   Since both actions are free and the orbits have half the dimension, the differentials of $\widetilde{\mu}^{\mathcal{L}}$ and $\widetilde{\mu}^{\mathcal{R}}$ are surjective, hence their images are open, and the fibers of $\widetilde{\mu}^{\mathcal{R}}$ are the orbits of the $\mathcal{L}$-action and vice versa.  From this one checks straightforwardly that $\Psi$ is a $K\times K$-equivariant diffeomorphism that intertwines the moment maps.  

To see that $\Psi$ is a symplectomorphism, observe that from the properties of $\Pi$ it follows at once that the fibers of $\Pi$ are transversal to the orbits (for both actions).  At any point $g$ of $G$ we can therefore write any tangent vector $X$ as a sum of a `vertical' part $X_v$, which is tangent to a fiber of $\Pi$, and a `horizontal' part $X_h$, tangent to the orbit of the $\mathcal{R}$-action.  Using the fact that the fibers of $\Pi$ are Lagrangian we have $$\omega_G(X_h+X_v,Y_h+Y_v)=\omega_G(X_h, Y_h+Y_v)-\omega_G(Y_h,X_v),$$ and likewise for the pull-back of the symplectic form on $T^*K$.  Since both $X_h$ and $Y_h$ are evaluations of vector fields generated by $\mathcal{R}$ at $g$, the contractions of the symplectic forms by them are determined by the moment maps $\widetilde{\mu}^{\mathcal{R}}$ and $\mu^{\mathcal{R}}$.  Since $\Psi$ intertwines these moment maps, the evaluations of both forms are identical, and the result follows.
\end{proof}

The Cartan decomposition (see e.g \cite[Theorem 6.31]{knapp}) says that every element $g$ in $G$ can uniquely be written as $g=ke^{i\lambda}$, with $k\in K$ and $\lambda \in \mathfrak{k}$ -- this gives in fact a diffeomorphism between $K\times \mathfrak{k}$ and $G$.  As a consequence we obtain a canonical projection $$\Pi_{\text{cd}}:G\to K:g\mapsto k$$ which is equivariant for $\mathcal{L}$ and $\mathcal{R}$.  It turns out that if we equip $G$ with the K\"ahler form obtained from any faithful representation by restricting the Euclidean K\"ahler form on $GL(N,\mathbb{C})\subset M_{N\times N}(\mathbb{C})$, we can use this $\Pi_{\text{cd}}$ to apply Lemma \ref{sympl}.  Indeed, we have

\begin{lemma}\label{sympl2}
The fibers of $\Pi_{\text{cd}}$ are Lagrangian for the K\"ahler form on $G$ inherited from $M_{N\times N}(\mathbb{C})$.
\end{lemma}
\begin{proof}
This follows from a simple direct computation: the K\"ahler metric on \linebreak $M_{N\times N}(\mathbb{C})$ can be written as \begin{equation}\label{metric}g_E(A,B)=\Tra (A\ B^*),\end{equation} and hence the K\"ahler form can be written as $$\omega_E(A,B)=\rea \left(\Tra\left( A\ (iB)^*\right) \right) = -\ima \left(\Tra\left( A\ B^*\right) \right).$$  
Tangent vectors to a fiber of $\Pi_{\text{cd}}$ at a point $ke^{i\lambda}$ can be written as $\frac{d}{dt}\big|_0 ke^{i(\lambda+t\nu)}$, with $\mu$ and $\nu$ elements of $\mathfrak{u}(N)$, i.e.\ anti-Hermitian matrices.  
We therefore need to evaluate \begin{equation}\begin{split}\label{nul}
\omega_E\left(\frac{d}{dt}\Big|_0  ke^{i(\lambda+t\nu)}  ,\frac{d}{ds}\Big|_0  ke^{i(\lambda+s\xi)}\right)
&=-\frac{\partial^2}{\partial t \partial s}\Big|_{(0,0)} \ima\Tra\left( ke^{i(\lambda+t\nu)}\left(ke^{i(\lambda+s\xi)}\right)^*\right) \\ 
&= -\frac{\partial^2}{\partial t \partial s}\Big|_{(0,0)} \ima\Tra\left( e^{i(\lambda+t\nu)}\ e^{i(\lambda+s\xi)}\right).
\end{split}\end{equation}
But for any anti-Hermitian $\alpha$ and $\beta$ one has that 
\begin{equation*}
\overline{\Tra \left( e^{i\alpha}\ e^{i\beta}\right)}
= \Tra \left( e^{i\alpha^t}\ e^{i\beta^t}\right)= \Tra \left(e^{i\alpha}e^{i\beta}\right).
\end{equation*}
Hence any such $\Tra\left(e^{i\alpha}\ e^{i\beta}\right)$ is always real, and therefore (\ref{nul}) vanishes.
\end{proof}

\begin{remark}

There are at least two obvious ways one can equip $G$ with a K\"ahler structure: besides the one obtained from restricting the ambient Euclidean K\"ahler metric through an affine embedding (which we use in this paper), one can choose an invariant metric on $\mathfrak{k}$ and use this to identify $\mathfrak{k}$ and $\mathfrak{k}^*$.  Using the Cartan decomposition this gives $K\times K$-equivariant diffeomorphisms $$G\cong K\times \mathfrak{k}\cong K\times \mathfrak{k}^*\cong T^*K.$$  
The canonical symplectic structure on $T^*K$ and the complex structure on $G$ combine to a K\"ahler structure, which is discussed in \cite{hall}.   Both of these methods depend on a choice, but  they are mutually exclusive.
\end{remark}

\section{Polar decompositions and moment maps for normal reductive monoids}\label{polar}
There is a related area where we shall need the polar decomposition.  Recall that any matrix $A$ in $M_{N\times N}(\mathbb{C})$ can be written as $$A=UP,\hspace{1cm}\text{with }\hspace{1cm} U\in U(N)\hspace{1cm}\text{and }\hspace{1cm}P=\sqrt{ {A}^* A}.$$  In this decomposition $P$ is of course unique, but if $A$ is not invertible $U$ is not (for invertible complex matrices the Cartan decomposition and the polar decomposition coincide).  Closely related to this is the fact that the moment map for the $\mathcal{R}$-action of $U(N)$ on $M_{N\times N}(\mathbb{C})$ (equipped with the Euclidean K\"ahler metric) is given by 
\begin{equation}\label{momentpolar}\mu(A)=i {A}^*A,\end{equation} where we identify $\mathfrak{u}(N)$ with $\mathfrak{u}^*(N)$ by means of $\langle A,B\rangle= -\Tra(AB)$, compatible with the metric (\ref{metric}) we have chosen on $M_{N\times N}(\mathbb{C})$.  
The polar decomposition allows us to write down a section (i.e.\ a right inverse) $s$ for $\mu$: if $B$ is in the image of $\mu$ (i.e.\ if $-iB$ is positive semidefinite), simply put  \begin{equation}\label{thesection}s(B)=\sqrt{-iB}.\end{equation}   The entire pre-image of $B$ under $\mu$ is then just the $\mathcal{L}$-orbit $U(N)s(B)$.  The section $s$ is continuous; moreover, it is smooth on $\mu(GL(n,\mathbb{C}))$.

These two items -- a section of the moment map and a description of the fibers of $\mu$ as $\mathcal{L}$-orbits -- are inherited by suitable submonoids of $M_{N\times N}(\mathbb{C})$.  Indeed, we have the following:

 \begin{lemma}\label{shuffle} Let $S$ be a normal submonoid of $M_{N\times N}(\mathbb{C})$ (with inclusion denoted by $\iota$)
  given as the closure of a reductive subgroup $H=L_{\mathbb{C}}$ of $GL(N,\mathbb{C})$, with $L=H\cap U(N)$ and moment map $\mu_{L}:S\rightarrow \mathfrak{l}^*$ for the $\mathcal{R}$-action of $L$ on $S$.
 The map \begin{equation}\label{allerlaatste}(d\iota)^*: (\mu\circ \iota)(S)\longrightarrow \mu_L(S)\end{equation} is a homeomorphism, moreover it is a diffeomorphism between $(\mu\circ \iota)(H)$ and $\mu_L(H)$.
 \end{lemma}
  We shall denote the inverse of this homeomorphism as $\eta$; it is again smooth on the relative interior of its domain.  Observe that we always have  \begin{equation}\label{triptic}(d\iota)^*\circ \mu\circ \iota = \mu_L\end{equation} from functoriality of the moment map.
 
 \begin{proof}
 We begin by restricting to $H\subset S$.  By using Lemma \ref{sympl} (valid because of Lemma \ref{sympl2}) and the Cartan decomposition, we get a diffeomorphism $$\chi:\mu_L(H)\overset{\cong}{\longrightarrow}\mathfrak{l}.$$  Using (\ref{momentpolar}) and the fact that the Cartan decomposition is functorial with respect to $\iota$, $$\mu_L(H) \rightarrow \mu(\iota(H))\subset \mathfrak{u}(N)^*: x\mapsto i e^{i 2(d\iota\ \circ\ \chi)(x)}$$ is an inverse to the restriction of (\ref{allerlaatste}) to $H$ and is a diffeomorphism onto its image -- by using the metric on $\mathfrak{u}(N)^*\cong \mathfrak{u}(N)$ one can think of this image as the graph of a smooth function from $\mu_L(H)$ to the orthogonal complement of $\mathfrak{l}\subset \mathfrak{u}^N$.  
 
By \cite[Lemma 4.10]{sjamaar} both $\mu$ and $\mu_L$ are proper, hence it follows from (\ref{triptic}) that (\ref{allerlaatste}) is proper as well.  Therefore the rest of the statement follows provided we can establish that (\ref{allerlaatste}) is one-to-one, since a proper continuous bijection between locally compact Hausdorff spaces is always a homeomorphism.
 
Since $S$ is normal it is a spherical $H\times H$-variety (where we use both $\mathcal{L}$- and $\mathcal{R}$-actions); it follows that symplectically it is a multiplicity-free space for the action of $L \times L$ (see e.g.\ \cite[\S 5.1]{brionmoment} or the discussion in \cite[\S 2]{knop}).  As a result any fiber of $\Phi_{L\times L}$ (the moment map of the $L\times L$ action composed with the projection to the positive Weyl chamber) consists of a single $L\times L$-orbit.  Since $\Phi_{L\times L}=(\Phi^{\mathcal{L}}_L,\Phi^{\mathcal{R}}_L)$ and  $\Phi^{\mathcal{L}}_L=\tau\circ\Phi^{\mathcal{R}}_L$ (where $\tau$ is the involution of the positive Weyl chamber) this implies that any fiber of $\mu_L$ is contained in such an $L\times L$-orbit.  On the other hand, since the $\mathcal{L}$- and $\mathcal{R}$-actions of $L$ on $S$ commute we know that any fiber of the moment map for one is a union of orbits for the other.  Hence for the fibers of $\mu_L$ we can focus our attention on the $\mathcal{R}$-action of $L$.

Let $b$ be an element of $S\setminus H$.  For any $g\in L$,  $\mu_L(bg^{-1})=\mu_L(b)$ if and only if $\Adj^*_g(\mu_L(b))=\mu_L(b)$.  The only way that (\ref{allerlaatste}) could fail to be one-to-one is if there existed such a choice of $b$ and $g$ with $\Adj^*_{\iota(g)}(\mu(\iota(b)))\neq \mu(\iota(b))$.
 Since we know this cannot happen for elements in $H$, it suffices by continuity to show that we can approach $b$ by elements in $H$ whose images under $\mu_L$ are still stabilized by the same $g$.  
 
 In order to establish this we need some results about $\Phi^{\mathcal{R}}_L(S)$ (or, equivalently, $\Phi_{L\times L}(S)$).  By \cite[Theorem 4.9]{sjamaar} we know that $\Phi^{\mathcal{R}}_L(S)$ is the cone generated by the highest weights for the representation of $H$ on the ring of regular functions of $S$.  In turn the latter is described by \cite[Corollary 27.17]{timashev} (see also \cite[Theorem 2]{vinberg}) -- it is the intersection of the positive Weyl chamber of $L$ with the convex cone $\mathcal{K}$ generated by all the weights of the representation of the maximal torus of $H$ on $\mathbb{C}^N$ induced by $\iota$.  Moreover this gives us a description of the $H\times H$ orbits of $S$: by \cite[Theorem 27.20]{timashev} they are in bijection with the faces of $\mathcal{K}$ whose interiors intersect the positive Weyl chamber of $L$.  In particular, this implies that if $s\in S$ is such that $\Phi^{\mathcal{R}}_L(s)$ is in the interior of $\mathcal{K}$, then $s\in H$.
 
 Now, since $\mathcal{K}$ is invariant under the action of the Weyl group of $L$, it contains its own projection onto $\mathfrak{z}(L)^*$, the dual of the Lie algebra of the center of $L$.  This projection is a cone in $\mathfrak{z}(L)^*$, and since $\Phi^{\mathcal{R}}_L(S)$ has to generate all of the dual of the Cartan of $L$ (see \cite[Theorem 2]{vinberg}), it has to contain elements in the interior of $\mathcal{K}$.  Therefore we can approximate any  $\Phi^{\mathcal{R}}_L(b)$ by elements in the interior of $\mathcal{K}$ by adding elements of $\mathfrak{z}(L)^*$, i.e.\ without changing the stabilizer of $\mu_L(b)$ under the coadjoint representation. Since $S$ is the closure of $H$ this means that we can indeed approximate $b$ by elements in $H$ with the stabilizer condition as stipulated above, which concludes the proof.\end{proof}
 
We can put this together with the section $s$ of $\mu$ as 
 \begin{center}
\begin{tikzpicture}
\matrix (n)[matrix of math nodes,  row sep=3em, column sep=2.5em, text height=1.5ex, text depth=0.25ex]
{S & & M_{N\times N}(\mathbb{C}) \\ & & \\ \mu_L(S) 
& & \mu(M_{N\times N}(\mathbb{C}))
\\ \mathfrak{l}^* & &  
\mathfrak{u}(N)^*.\\};
\path[right hook->] 
(n-1-1) edge node[auto] {$\iota$} (n-1-3);
\path[->]
(n-1-1)
 edge  node[auto] {$\mu_L$} 
(n-3-1);
\path[->] 
(n-3-1) edge [bend left=15]  node[auto] {$\eta$} (n-3-3);
\path[->] 
(n-4-3) edge node[auto] {$(d\iota)^*$} (n-4-1);
\path[->] 
(n-1-3) edge  node[auto] {$ \mu $} (n-3-3);
\path[->] 
(n-3-1) edge [bend left=30] node[auto] {$s_L$}(n-1-1);
\path[right hook->] 
(n-3-1) edge (n-4-1);
\path[right hook->] 
(n-3-3) edge (n-4-3);
\path[->] 
(n-3-3) edge [bend left=30] node[auto] {$ s $}(n-1-3);
\end{tikzpicture}
\end{center}

\begin{lemma}\label{sectionconstr}
The composition $s_L:=s\circ \eta$ takes values in $S$, and can therefore be understood as a section of $\mu_L$. 
\end{lemma}
\begin{proof}
By continuity it suffices again to check this on $\mu_L(H)$.  The section $s$ is characterized by the fact that the polar decomposition of any point in its image can be chosen to have trivial unitary part.  Now, if an orbit of the $\mathcal{L}$-action of $U(N)$ on $GL(N,\mathbb{C})$ meets $H$ (say in an element $h$), then the unique element in this $\mathcal{L}$-orbit whose polar decomposition has trivial unitary part has to be contained in $H$ itself.  Indeed, use the Cartan decomposition for $H$ to write $h=le^{i\lambda}$.  Then clearly also $e^{i\lambda}$ is contained in both the $\mathcal{L}$-orbit of $U(N)$ and in $H$.  As the Cartan decomposition is preserved by $\iota$, this element will have trivial unitary part for the Cartan decomposition of $GL(N,\mathbb{C})$, as well as for the polar decomposition (since the Cartan decomposition and the polar decomposition coincide for invertible matrices).  Hence $s_L$ takes values in $S$, and since $\eta$ is a section of $(d\iota)^*$ and $s$ is a section of $\mu$ it follows from (\ref{triptic}) that $s_L$ is a section of $\mu_L$.
\end{proof}

This section $s_L$ is again continuous and smooth on the relative interior of its domain.

\begin{lemma}\label{crucial}
Let $S$ be a normal closed submonoid of $M_{N\times N}(\mathbb{C})$ as in Lemma \ref{shuffle}.  
Then the fibers of the moment map $\mu_L$ for the $\mathcal{R}$-action of $L$ on $S$ are the orbits for the $\mathcal{L}$-action of $L$ (and vice-versa).
\end{lemma}
\begin{proof}
From Lemma \ref{shuffle} and its proof we know that any fiber of $\mu_L$ is contained in an $L\times L$ orbit and gets mapped by $\iota$ to a single fiber of the moment map of the $\mathcal{R}$-action of $U(N)$ on $M_{N\times N}(\mathbb{C})$.
On the other hand we know that the fibers of $\mu$ are exactly the orbits of the $\mathcal{L}$-action of $U(N)$.  It now suffices to remark that, for a point of $S$ regarded as sitting in $M_{N\times N}(\mathbb{C})$, the intersection of its $L\times L$-orbit  with its $\mathcal{L}$-orbit for $U(N)$ consists exactly of  its $\mathcal{L}$-orbit for $L$.
\end{proof}

\providecommand{\bysame}{\leavevmode\hbox to3em{\hrulefill}\thinspace}
\providecommand{\href}[2]{#2}


\begin{thebibliography}{LMTW98}

\bibitem[AB04]{alexeevbrion}
Valery Alexeev and Michel Brion, \emph{Stable reductive varieties. {I}.
  {A}ffine varieties}, Invent. Math. \textbf{157} (2004), no.~2, 227--274.

\bibitem[BCS05]{BCS}

Lev~A. Borisov, Linda Chen, and Gregory~G. Smith, \emph{The orbifold {C}how
  ring of toric {D}eligne-{M}umford stacks}, J. Amer. Math. Soc. \textbf{18}
  (2005), no.~1, 193--215 (electronic).

\bibitem[Bri87]{brionmoment}
Michel~Brion, \emph{Sur l'image de l'application moment}, S\'eminaire d'alg\`ebre Paul Dubreil et Marie-Paule Malliavin (Paris, 1986), 177--192, 
Lecture Notes in Mathematics, \textbf{1296}, Springer, Berlin, 1987. 

\bibitem[BGL02]{kahlercuts}
D.~Burns, V.~Guillemin, and E.~Lerman, \emph{Kaehler cuts}, 2002, \href{http://arxiv.org/abs/math/0212062}{arXiv:math/0212062}.

\bibitem[CLS12]{coxbigbook}                            
 D.~A. Cox, J.~B. Little and  H.~K. Schenck,    
\emph{Toric varieties}, 
Graduate Studies in Mathematics, vol.~\textbf{124}, American Mathematical Society, Providence, RI, 2011. 

\bibitem[Cox95]{cox}

David~A. Cox, \emph{The homogeneous coordinate ring of a toric variety}, J.
  Algebraic Geom. \textbf{4} (1995), no.~1, 17--50.

\bibitem[Del88]{delzant}
Thomas Delzant, \emph{Hamiltoniens p\'eriodiques et images convexes de
  l'application moment}, Bull. Soc. Math. France \textbf{116} (1988), no.~3,
  315--339.

\bibitem[DS10]{andrews}
A.~S. Dancer and A.~F. Swann, \emph{Non-abelian cut constructions and
  hyperk\"ahler modifications}, Rend. Semin. Mat. Univ. Politec. Torino
  \textbf{68} (2010), no.~2, 157--170.

\bibitem[EG98]{algcuts}
Dan Edidin and William Graham, \emph{Algebraic cuts}, Proc. Amer. Math. Soc.
  \textbf{126} (1998), no.~3, 677--685.

\bibitem[FMN10]{FMN}
Barbara Fantechi, Etienne Mann, and Fabio Nironi, \emph{Smooth toric
  {D}eligne-{M}umford stacks}, J. Reine Angew. Math. \textbf{648} (2010), 201--244.

\bibitem[GJS02]{implosion}
Victor Guillemin, Lisa Jeffrey, and Reyer Sjamaar, \emph{Symplectic implosion},
  Transform. Groups \textbf{7} (2002), no.~2, 155--184.
  
\bibitem[GS84]{guilsternconf}
Victor Guillemin and Shlomo Sternberg, 
\emph{A normal form for the moment map}, Differential geometric methods in mathematical physics (Jerusalem, 1982), 161--175, 
Math. Phys. Stud., \textbf{6}, Reidel, 1984.

\bibitem[GS90]{inphysics}
\bysame, \emph{Symplectic techniques in physics},
  second ed., Cambridge University Press, Cambridge, 1990.

\bibitem[Gom95]{gompf}  
Robert E.  Gompf,
\emph{A new construction of symplectic manifolds},
Ann. of Math. (2) \textbf{142} (1995), no.~3, 527--595. 

\bibitem[Hal97]{hall}
Brian~C. Hall, \emph{Phase space bounds for quantum mechanics on a compact
  {L}ie group}, Comm. Math. Phys. \textbf{184} (1997), no.~1, 233--250.

\bibitem[Hau98]{hausel}
Tam{\'a}s Hausel, \emph{Compactification of moduli of {H}iggs bundles}, J.
  Reine Angew. Math. \textbf{503} (1998), 169--192.
  
 \bibitem[HS02]{haussturm}                           
T. Hausel and B. Sturmfels,  
\emph{Toric hyperk\"ahler varieties},       
Documenta Math. \textbf{7} (2002), 495--534.

\bibitem[HNP94]{otherconv}
Joachim Hilgert, Karl-Hermann Neeb, and Werner Plank, \emph{Symplectic
  convexity theorems and coadjoint orbits}, Compositio Math. \textbf{94}
  (1994), no.~2, 129--180.

\bibitem[Kir84]{kirwanconv}
Frances Kirwan, \emph{Convexity properties of the moment mapping. {III}},
  Invent. Math. \textbf{77} (1984), no.~3, 547--552.

\bibitem[Kna02]{knapp}
Anthony~W. Knapp, \emph{Lie groups beyond an introduction}, second ed.,
  Progress in Mathematics, vol. \textbf{140}, Birkh\"auser Boston Inc., Boston, MA,
  2002.
  
\bibitem[Kno11]{knop}
Friedrich Knop, \emph{Automorphisms of multiplicity free {H}amiltonian
  manifolds}, J. Amer. Math. Soc. \textbf{24} (2011), no.~2, 567--601.

\bibitem[Ler95]{lerman}
Eugene Lerman, \emph{Symplectic cuts}, Math. Res. Lett. \textbf{2} (1995),
  no.~3, 247--258.

\bibitem[LM12]{lermanorbi}
Eugene Lerman and Anton Malkin, \emph{Hamiltonian group actions on symplectic
  {D}eligneÐ-{M}umford stacks and toric orbifolds}, Adv. Math. \textbf{229}
  (2012), no.~2, 984--1000.

\bibitem[LMTW98]{convbycuts}
Eugene Lerman, Eckhard Meinrenken, Sue Tolman, and Chris Woodward,
  \emph{Nonabelian convexity by symplectic cuts}, Topology \textbf{37} (1998),
  no.~2, 245--259.

\bibitem[Los09]{losev}
Ivan~V. Losev, \emph{Proof of the {K}nop conjecture}, Ann. Inst. Fourier
  (Grenoble) \textbf{59} (2009), no.~3, 1105--1134.

\bibitem[LR01]{liruan}
An-Min Li and Yongbin Ruan, \emph{Symplectic surgery and {G}romov-{W}itten
  invariants of {C}alabi-{Y}au 3-folds}, Invent. Math. \textbf{145} (2001),
  no.~1, 151--218.

\bibitem[LT97]{lermantolman}
Eugene Lerman and Susan Tolman, \emph{Hamiltonian torus actions on symplectic
  orbifolds and toric varieties}, Trans. Amer. Math. Soc. \textbf{349} (1997),
  no.~10, 4201--4230.
  
\bibitem[Mar85]{marle}
Charles-Michel Marle, 
\emph{Mod\`ele d'action hamiltonienne d'un groupe de Lie sur une vari\'et\'e symplectique}, Rend. Sem. Mat. Univ. Politec. Torino \textbf{43} (1985), no.~2, 227--251.

\bibitem[Mar08]{me}
Johan Martens, \emph{Equivariant volumes of non-compact quotients and instanton
  counting}, Comm. Math. Phys. \textbf{281} (2008), no.~3, 827--857.

\bibitem[Mei98]{meinrenken}
Eckhard Meinrenken, \emph{Symplectic surgery and the {${\rm Spin}^c$}-{D}irac
  operator}, Adv. Math. \textbf{134} (1998), no.~2, 240--277.

\bibitem[MT11]{us}
Johan Martens and Michael Thaddeus, \emph{Compactifications of reductive groups
  as moduli stacks of bundles}, 2011, \href{http://arxiv.org/abs/1105.4830}{arXiv:1105.4830}.

\bibitem[MW08]{extendedcut}
Ciprian Manolescu and Christopher Woodward, \emph{The symplectic cut of the
  extended moduli space}, 2008, unpublished.
  
\bibitem[MW12]{extendedfloer}
\bysame, \emph{Floer homology on the extended moduli space}, Perspectives in
  A\-na\-ly\-sis, Geometry, and Topology (Ilia Tenberg, Burglind J\"oricke, and
  Mikael Passare, eds.), Progress in Mathematics, vol.~\textbf{296}, Birkhauser, 2012, pp. 283--329,
 \href{http://arxiv.org/abs/0811.0805}{arXiv:0811.0805}.
  
\bibitem[Par09]{paradan}
Paul-{\'E}mile Paradan, \emph{Formal geometric quantization}, Ann. Inst.
  Fourier (Grenoble) \textbf{59} (2009), no.~1, 199--238.

\bibitem[Pez10]{pezzini}
Guido Pezzini, \emph{Lectures on spherical and wonderful varieties}, Actions
  hamiltoniennes: invariants et classification (Michel Brion and Thomas
  Delzant, eds.), Les cours du C.I.R.M., tome \textbf{1}, num\'ero 1, {C.I.R.M.}, 2010, pp.~33--53,
  Available from \url{http://ccirm.cedram.org/ccirm-bin/feuilleter}.
  
\bibitem[Sja98]{sjamaar}
Reyer Sjamaar, \emph{Convexity properties of the moment mapping re-examined},
  Adv. Math. \textbf{138} (1998), no.~1, 46--91.

\bibitem[Tha96]{gitflips}
Michael Thaddeus, \emph{Geometric invariant theory and flips}, J. Amer. Math.
  Soc. \textbf{9} (1996), no.~3, 691--723.
  
\bibitem[Tim11]{timashev}
D.A. Timashev,
\emph{Homogeneous spaces and equivariant embeddings},
Encyclopaedia of  Mathematical Sciences, vol.~\textbf{138}, 
Springer, Heidelberg, 2011.

\bibitem[Tol98]{tolman}
Susan Tolman, \emph{Examples of non-{K}\"ahler {H}amiltonian torus actions},
  Invent. Math. \textbf{131} (1998), no.~2, 299--310.

\bibitem[Vin95a]{vinberg2}
Ernest~B. Vinberg, \emph{The asymptotic semigroup of a semisimple {L}ie group},
  Semigroups in algebra, geometry and analysis ({O}berwolfach, 1993), de
  Gruyter Exp. Math., vol.~\textbf{20}, de Gruyter, Berlin, 1995, pp.~293--310.

\bibitem[Vin95b]{vinberg}
\bysame, \emph{On reductive algebraic semigroups}, Lie groups and {L}ie
  algebras: {E}. {B}. {D}ynkin's {S}eminar, Amer. Math. Soc. Transl. Ser. 2,
  vol. \textbf{169}, Amer. Math. Soc., Providence, RI, 1995, pp.~145--182.
  
\bibitem[Wei01]{weitsman}
Jonathan Weitsman, \emph{Non-abelian symplectic cuts and the geometric
  quantization of noncompact manifolds}, Lett. Math. Phys. \textbf{56} (2001),
  no.~1, 31--40,  EuroConf{\'e}rence Mosh{\'e} Flato
  2000, Part I (Dijon).

\bibitem[Woo96]{woodward}
Chris Woodward, \emph{The classification of transversal multiplicity-free group
  actions}, Ann. Global Anal. Geom. \textbf{14} (1996), no.~1, 3--42.

\bibitem[Woo98]{nonkahler}
\bysame, \emph{Multiplicity-free {H}amiltonian actions need not be {K}\"ahler},
  Invent. Math. \textbf{131} (1998), no.~2, 311--3198.

\end{thebibliography}
\end{document}